\documentclass[a4paper,12pt]{amsart}
\usepackage[utf8]{inputenc}

\usepackage[T1]{fontenc}
\usepackage[english]{babel}
\usepackage{color}
\usepackage{amsmath} 
\usepackage{amssymb} 
\usepackage{amsthm}
\usepackage{graphicx}
\usepackage[colorlinks=true, linkcolor=blue]{hyperref}

\usepackage{geometry}

\usepackage{lmodern}
\usepackage{microtype}

\usepackage{subfiles}

\usepackage{soul}

\usepackage[colorinlistoftodos,prependcaption,textsize=tiny]{todonotes}

\newcommand {\R} {\mathbb{R}} 
\newcommand {\Z} {\mathbb{Z}}
\newcommand {\T} {\mathbb{T}}

\newcommand {\p} {\partial}

\newcommand{\bbN}{\mathbb{N}}

\newcommand{\bbR}{\mathbb{R}}

\newcommand{\bbT}{\mathbb{T}}

\newcommand{\bbZ}{\mathbb{Z}}

\newcommand{\eps}{\varepsilon}

\theoremstyle{plain}
\newtheorem{thm}{Theorem}[section]
\newtheorem*{thm*}{Theorem}
\newtheorem{prop}{Proposition}[section]

\newtheorem*{cor*}{Corollary}

\theoremstyle{plain}
\newtheorem{theorem}{Theorem}
\newtheorem{lemma}[theorem]{Lemma}
\newtheorem{pro}[theorem]{Proposition}

\theoremstyle{definition}

\date{\today}

\begin{document}

\title[MHD Stability Threshold]{Sobolev Stability for the 2D MHD Equations in the Non-Resistive Limit }

\begin{abstract}
    In this article, we consider the stability of the 2D magnetohydrodynamics (MHD) equations close to a combination of Couette flow and a constant magnetic field. We consider the ideal conductor limit for the case when viscosity $\nu$ is larger than resistivity $\kappa$, $\nu\ge \kappa>0$. For this regime, we establish a bound on the Sobolev stability threshold. Furthermore, for $\kappa\le \nu^3$ this system exhibits instability, which leads to norm inflation
    of size $\nu \kappa^{-\frac 1 3 }$. 
\end{abstract}
\author{Niklas Knobel}
\address{Karlsruhe Institute of Technology, Englerstraße 2,
  76131 Karlsruhe, Germany}
  \email{niklas.knobel@kit.edu}


\maketitle
\setcounter{tocdepth}{1}
\tableofcontents

\section{Introduction }
The equations of magnetohydrodynamics (MHD) 
\begin{align}
  \label{MHD}
  \begin{split}
    \partial_t V + V\cdot \nabla V+ \nabla \Pi  &= \nu \Delta  V + B\cdot\nabla B, \\
    \partial_t B + V\cdot\nabla B &= \kappa\Delta   B +B\cdot\nabla V, \\
    \nabla\cdot V=\nabla\cdot B  &= 0,\\
    (t,x,y) &\in \R^+ \times\bbT\times \R=: \Omega ,
  \end{split}
\end{align}
model the evolution of a magnetic field $B:\Omega \to \bbR^2 $ interacting with the velocity $V:\Omega \to \bbR^2 $ of a conducting fluid. The MHD equations are a common model used in astrophysics, planetary magnetism and controlled nuclear fusion \cite{davidson_2016}. The quantities $\nu, \kappa \ge 0$ correspond to fluid viscosity and magnetic resistivity. The pressure $\Pi: \Omega \to \R $ ensures that the velocity remains divergence-free.  A fundamental problem of fluid dynamics and plasma physics is the stability and long-time behavior of solutions to equation \eqref{MHD} and in particular stability of specific solutions. We consider the combination of an affine shear flow, called Couette flow, and a constant magnetic field:
\begin{align*}
    V_s&=ye_1,\\
    B_s&=\alpha e_1 .
\end{align*}
In particular, the solution combines the effects of mixing due to shear and coupling by the magnetic field. The Couette flow mixes any perturbation, which leads to increased dissipation rates, called enhanced dissipation, and stabilizes the equation. The coupling with a constant magnetic field propagates this mixing to magnetic perturbations. However, the magnetic field weakens the mixing, especially if viscosity is larger than resistivity, inviscid damping gets counteracted by algebraic growth for specific time regimes.

In the related case of the Navier-Stokes equation, that is when no magnetic field is present, one observes turbulent solutions as viscosity reaches small values. In contrast, the linearized problem around Couette flow is stable for all values of the viscosity. These phenomena are known as the Sommerfeld paradox \cite{li2011resolution} and highlight instability due to nonlinear effects. In \cite{bedrossian2015inviscid,dengmasmoudi2018,dengZ2019,bedrossian2013asymptotic,ionescu2020nonlinear} various authors show sharp stability in Gevrey 2 spaces (spaces between $C^\infty$ and analytic). The nonlinear instability can be suppressed by the viscosity for initial data sufficiently small in Sobolev spaces, ensuring stability \cite{bedrossian2016sobolev,masmoudi2022stability,bedrossian2017stability}. 

When considering the MHD equations without Couette flow, the constant magnetic field stabilizes the equation. The dynamics of small initial perturbations of the ideal MHD equation around a strong enough magnetic field is close to the linearized system  \cite{bardos1988longtime}. For stability in several dissipation regimes we refer to \cite{wei2017global, ren2014global, he2018global, ren2014global, schmidt1988magnetohydrodynamic, cobb2023elsasser,kozono1989weak} and references therein. However, global in time wellposedness for the non-resistive case is still open (see the discussion in \cite{cobb2023elsasser}). Furthermore, a shear flow leads to qualitatively different behavior and instabilities \cite{hughes2001instability,hussain2018instability}.
 
Recently, the MHD equation around Couette flow has gathered significant interest \cite{liss2020sobolev, knobel2023echoes, zhao2023asymptotic, Dolce, knobel2023sobolev}. Already on a linear level, the behavior of the MHD changes for different values of  $\nu$ and $\kappa$. In \cite{liss2020sobolev} Liss proved the first stability threshold for the MHD equations. He considered the full dissipative regime of $\kappa =\nu>0$ and proved the stability of the three-dimensional MHD equation for initial data which is sufficiently small in Sobolev spaces. For the analogous two-dimensional problem,  Dolce \cite{Dolce} proved stability in the more general setting of $0<\kappa^3\lesssim \nu\le \kappa $. In \cite{knobel2023sobolev} Zillinger and the author considered the case of only horizontal resistivity and full viscosity and established stability for small data in Sobolev spaces. For the regime of vanishing viscosity $\nu=0$ and non-vanishing resistivity $\kappa>0$, in \cite{knobel2023echoes} we constructed a linear stability and instability mechanism around nearby traveling waves in Gevrey $2$ spaces. In a corresponding nonlinear stability result,  Zhao and Zi \cite{zhao2023asymptotic} proved the almost matching nonlinear result of Gevrey $\sigma$ stability for $1\le \sigma <2$ and for sufficiently small perturbations. 
 
The results mentioned above on stability around Couette flow focus on the setting when resistivity is larger than viscosity $\nu \le \kappa$.  Indeed in the setting $\nu>0$ and $\kappa=0$, the magnetic effects dominate leading to a linear instability mechanism and thus a growth of the magnetic field by $\nu t$ for specific initial data \cite{knobel2023sobolev}. 

In this paper, we consider the setting $0<\kappa \le \nu$. In particular, this also includes the non-resistive limit $\kappa \downarrow 0 $ independent of $\nu$. To the author's knowledge the stability of the regime $\kappa < \nu$ has not previously been studied for the MHD equation around Couette flow. To state the main result, we define the perturbative unknowns 
\begin{align*}
    v(x,y,t)&= V(x+yt,y,t )- V_s, \\
    b(x,y,t)&= B(x+yt,y,t )- B_s,
\end{align*}
where the change of variables $x\mapsto x+yt$ follows the characteristics of the Couette flow. For these unknowns, equation \eqref{MHD} becomes 
\begin{align}
\begin{split}
    \partial_t v + v_2 e_1 - 2\partial_x \Delta^{-1}_t  \nabla_t v_2  &=  \nu  \Delta_t v+ \alpha \partial_x b  + b\nabla_t b- v\nabla_t v-\nabla_t \pi , \\
    \partial_t b - b_2 e_1 \qquad \qquad \quad \quad  \ &= \kappa \Delta_t   b+ \alpha \partial_x v  +b\nabla_t v -v\nabla_t b,\\
    \nabla_t\cdot v=\nabla_t\cdot b &= 0.\label{eq:vb}
\end{split}
\end{align}
Due to the change of variables the spatial derivatives become time-dependent, i.e. $\p_y^t= \p_y-t\p_x$, $\nabla_t = (\p_x,\p_y^t)^T$ and $\Delta_t = \p_x^2 +(\p_y^t)^2 $.

For equation \eqref{eq:vb} we establish Lipschitz stability for initial data which is sufficiently small in Sobolev spaces, in the sense that there exists a bound on the initial data $\eps_0=\eps_0(\nu,\kappa)$ and a Lipschitz constant $L=L(\nu,\kappa) $ such that for initial data which satisfies
\begin{align*}
        \Vert (v,b)_{in}\Vert_{H^N} = \eps \le  \eps_0,
\end{align*}
the corresponding solution is globally bounded in time by 
\begin{align*}
        \Vert (v,b)(t)\Vert_{H^N}\le L \eps. 
\end{align*}
For the non-resistive case, $\kappa=0$, global wellposedness is an open problem and so Lipschitz stability in Sobolev spaces is unclear. Thus, naturally the question arises, which $\eps_0$ and $L$ are optimal and how they behave in the limit $\nu ,\kappa \downarrow 0$.\\
We denote a Sobolev stability threshold as $\gamma_1, \gamma_2 \in \bbR$, such that for $\eps_0= c_0 \nu^{\gamma_1}\kappa^{\gamma_2}$ with small $c_0>0$  we obtain 
\begin{align*}
    \Vert (v,b)_{in}\Vert_{H^N}  &\le  c_0  \nu^{\gamma_1}\kappa^{\gamma_2}\to \text{ stability,} \\
    \Vert (v,b)_{in}\Vert_{H^N}  &\gg c_0  \nu^{\gamma_1}\kappa^{\gamma_2}\to \text{possible instability}. 
\end{align*}
This extends the common convention in the field (eg. see \cite{bedrossian2016sobolev}) to allow for two independent parameters $\nu$ and $\kappa$. In particular, it agrees with the common convention when restricting to the case $\nu \approx \kappa$. It allows us to discuss cases where $\kappa$ tends to zero much quicker than $\nu$. Establishing a possible instability is highly nontrivial since for the nonlinear setting it is difficult to construct solutions that exhibit norm inflation. To the author's knowledge, there does not exist any nonlinear instability result for the MHD equation around Couette flow in Sobolev spaces. 

For accessibility and simplicity of notation, we state our main result as the following theorem (see Theorem \ref{thm:comp} for a detailed description). 
\begin{thm}\label{thm:simp}
    Consider $\alpha >\tfrac 1 2 $, $N\ge 5$ and a small enough constant $c_0=c_0(\alpha )>0$. Let  $0<\kappa \le \nu \le\tfrac 1{ 40}  (1-\tfrac 1{2\alpha })^{\frac 65} $, then we obtain Sobolev stability for initial data which is sufficiently small in Sobolev spaces, where the estimates qualitatively differ for the regimes $\kappa\gtrsim\nu^3$ and $\kappa \lesssim \nu^3$. More precisely: 
    \begin{itemize}
        \item In the regime of $\nu^3 \lesssim \kappa$, for all initial data which satisfy 
    \begin{align*}
        \Vert(v, b)_{in} \Vert_{H^N }= \eps \le c_0  \nu^{\frac 1 {12} } \kappa^{\frac 1 2},
    \end{align*}
    the global in time solution $(v,b)$ of \eqref{eq:vb} satisfies the Lipschitz bound 
    \begin{align*}
        \sup_{t>0}\Vert (v, b)(t)\Vert_{H^N }\lesssim  \eps . 
    \end{align*}
    \item  In the regime of $\nu^3 \gtrsim \kappa$, for all initial data which satisfy 
    \begin{align*}
        \Vert(v, b)_{in} \Vert_{H^N }= \eps \le c_0  \nu^{-\frac {11} {12} } \kappa^{\frac 56},
    \end{align*}
    the global in time solution $(v,b)$ of \eqref{eq:vb} satisfies the Lipschitz bound 
    \begin{align*}
        \sup_{t>0}\Vert (v, b)(t)\Vert_{H^N }\lesssim \nu \kappa^{-\frac 1 3 }   \eps . 
    \end{align*}        
\end{itemize}
In particular, we obtain Lipschitz stability for the Lipschitz constant $L\approx \max( 1, \nu \kappa^{-\frac 1 3 } ) $ for the smallness parameter $\eps_0 \approx\min ( \nu^{\frac 1 {12} } \kappa^{\frac 1 2},\nu^{-\frac {11} {12} } \kappa^{\frac 56})$. 
\end{thm}

In the proof, we employ an energy method similar to \cite{bedrossian21, masmoudi2023asymptotic, zillinger2020boussinesq, Dolce, knobel2023sobolev}. In the following, we outline the main challenges and novelties of the proof:
\begin{itemize}
\item The imbalance of resistivity $\kappa$ and viscosity $\nu$ yields  two cases $\nu^3 \lesssim  \kappa$ and $\nu^3 \gtrsim \kappa$ (or equivalently $1\lesssim \nu \kappa^{-\frac 13 }$ or $1\gtrsim \nu \kappa^{-\frac 13 }$). These cases give different values for $L$, namely $1$ and $\nu \kappa^{-\frac 1 3 } $. 

\item We consider the case  $\nu^3 \gtrsim \kappa$. On certain time scales the viscosity is so strong that fluid effects get suppressed while the effects of the magnetic field dominate. Thus, the term $\p_t b= e_1 b_2$ in \eqref{MHD} generates algebraic growth in specific regimes (see Subsection \ref{sec:svis}). Estimating this linear effect yields the norm inflation by $L=\nu \kappa^{-\frac 1 3}$. The algebraic growth appears on different time scales depending on the frequency, a precise estimate of the nonlinear terms is necessary.

\item For the case  $\nu^3\lesssim  \kappa $ the algebraic growth is bounded by a finite constant. In the subcase  $\nu = \kappa $ the sum of the threshold parameters is $\gamma_1 + \gamma_2 = \tfrac 7{12}$ which is a slight improvement over  $\tfrac 2 3$ in \cite{Dolce}. 

\item In the proof of Theorem \ref{thm:simp} we perform a low and high frequency decomposition $a= a_{hi} +a_{low}$. For high frequencies, the nonlinear term consist of $a_{low} \nabla_t a_{hi}$, called transport term and $a_{hi } \nabla_t a$, called reaction term  (including $hi-hi$ interactions). Compared to the Navier-Stokes equation, in the case of the MHD equation, it is vital to bound the transport term precisely. In particular, for $\kappa \lesssim \nu^3 $ the previously mentioned algebraic growth affects the estimate of the transport term strongly. 

\item The threshold is determined by the nonlinear term $v\nabla_t  b= \Lambda_t^{-1}\nabla^\perp p_1  \nabla b$ acting on $b$ in \eqref{eq:vb}, for the natural unknown $p_1= \Lambda^{-1}_t \nabla^\perp v$ (which we discuss later in more detail). In our estimates we rely on two stabilizing effects, the strong viscosity of $v$ and the $\Lambda^{-1}_t$ in front of $p_1$. For the nonlinear term $v\nabla_t  b$ both effects fall onto $v$. Due to the weaker integrability of the $b$ this term determines the threshold after integrating in time.

\end{itemize}
With the main challenges in mind, let us comment on the results:
\begin{itemize}
    \item The size of the constant magnetic field  $\alpha>\tfrac 1 2 $ results in a strong interaction between $v$ and $b$. Due to this interaction, the decay in $v$ and growth in $b$ are in balance (see Lemma \ref{pro:lin}). Constants may depend on $\alpha$ and degenerate as $\alpha \downarrow \tfrac 1 2 $. For example we obtain $\lim_{\alpha \downarrow \frac 1 2 } c_0(\alpha )=0$. 

    \item  Figure \ref{fig:ViscousMHD} shows which areas stability has been proven. The graphic shows only qualitative behavior and after rescaling we obtain the same graphic. The resistivity $\kappa$ is on the vertical axis and the viscosity $\nu$ is on the horizontal axis. We prove stability for the regime $0<\kappa \le \nu $, which we divide into two segments: $\nu^3 \lesssim \kappa $ in orange and $\nu^3 \gtrsim\kappa $ in red. In \cite{Dolce} Dolce considered the regime of $0<(\tfrac {16}\alpha \kappa)^3 \le \nu\le \kappa $, which is in blue. The authors of \cite{zhao2023asymptotic} considered the line $\nu=0$ which is in purple. The black line corresponds to $\nu= \kappa>0$ of \cite{liss2020sobolev}.
\begin{figure} 
\begin{center} 
\includegraphics[scale=0.25]{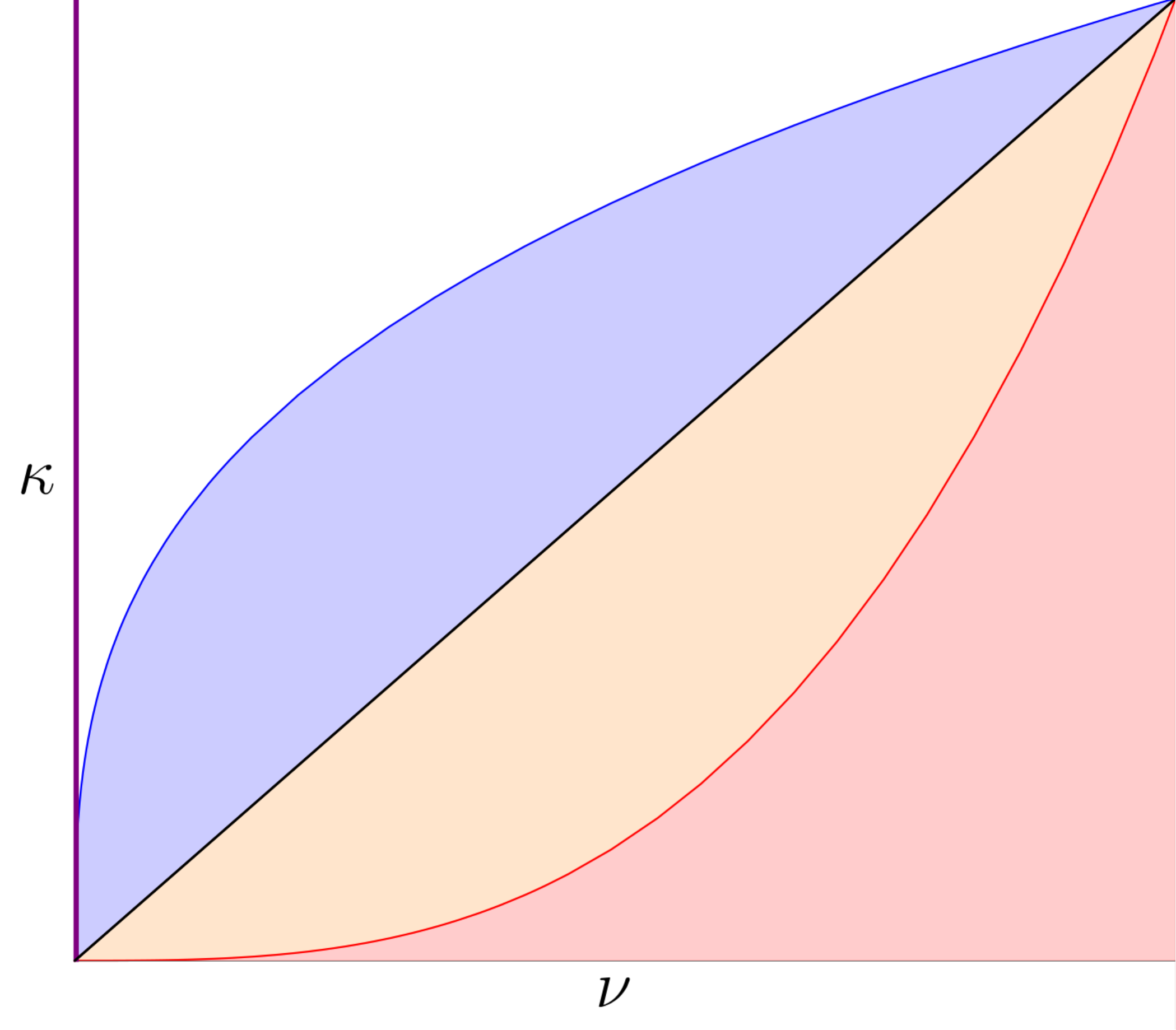}    
\caption{Sketch of areas with results for stability.  } 
\label{fig:ViscousMHD} 
\end{center} 
\end{figure}
    
    Stability for the regimes $0< \nu \le (\tfrac {16} \alpha  \kappa)^3 $, $\kappa =0<\nu $ and $\kappa=\nu=0$  remain open. For the set $0< \nu \ll \kappa^3 $, we expect that an adjusted application of the methods used in this article yield stability. We expect stability for the case $\kappa=0$ and $0<\nu$ to be very difficult since we obtain linear growth for the $p_2$ variable. For  $\Lambda_t^{-1} p_2$ we obtain linear stability but then there is no time decay in the magnetic field and so we lack an important stabilizing effect. In the inviscid case, $\kappa =\nu =0$ the linearized system is stable in the $p$ variables. However, due to the lack of dissipation, it is very challenging to bound the nonlinear terms. 
    
    \item Our threshold consists of parameters $\gamma_1$ and $\gamma_2$. An alternative notation is to impose the relation  $\nu \approx \kappa^\delta$ for some $0\le\delta \le 1  $. 
    With that convention we obtain stability if $\eps\le c_0 \kappa^{\gamma(\delta)}  $ for 
    \begin{align*}
        \gamma&=   \left\{
        \begin{array}{cc}
             \frac 1 2 +\frac \delta {12}& \delta \ge \tfrac 1 3   \\
             \tfrac 5 6- \tfrac {11}{12}\delta  & \text{ otherwise}.
        \end{array}
        \right. 
    \end{align*}

\end{itemize}

The remainder of this article is structured as follows: 
\begin{itemize}
    \item In Section \ref{sec:linbeh} we discuss the  linearized system. We identify two different time regions where ``circular movement'' or ``strong viscosity'' determine the linearized behavior. We estimate both effects separately and then establish the estimates for the linearized system. 
    \item In Section \ref{sec:StabThres} we prove the main theorem. We employ a bootstrap approach, where we control errors in Proposition \ref{prop:errors}. The main difficulty is to bound the linear growth and the nonlinear effect of $v\nabla_t b$ acting on $b$.
\end{itemize}

\subsection*{Notations and Conventions}
\label{sec:notation}
For $a,b \in \R$ we denote their minimum and maximum as 
\begin{align*}
    \min(a,b)&=a\wedge b,\\
    \max(a,b)&=a\vee b.
\end{align*}
We write $f\lesssim g $ if $f\le C g $ for a constant $C$ independent of $\nu$ and $\kappa$. Furthermore, we write $f\approx g $ if $f\lesssim g$ and $g\lesssim f $. We denote the Lebesgue spaces $L^p=L^p(\T\times \R)$ and the Sobolev spaces $H^N= H^N(\T\times \R )$ for some $N\in \bbN$. For time-dependent functions, we denote $L^p H^s=L^p_t H^s$ as the space with the norm 
\begin{align}
    \Vert f \Vert_{L^pH^s}&= \left\Vert \Vert f\Vert_{H^s(\T\times \R)}\right \Vert_{L^p(0,T)},\label{eq:lebT}
\end{align}
where omit writing the $T$. We write the time-dependent spatial derivatives
\begin{align*}
    \p_y^t&= \p_y-t\p_x,\\
    \nabla_t &= (\p_x,\p_y^t)^T,\\
    \Delta_t &= \p_x^2 +(\p_y^t)^2,
\end{align*}
and the half Laplacians as 
\begin{align*}
    \Lambda &= (-\Delta)^{\frac 1 2 },\\
    \Lambda_t&= (-\Delta_t)^{\frac 1 2 }.
\end{align*}
The function $f\in H^N $ is decomposed into its $x$ average and the orthogonal complement
\begin{align*}
    f_=(y)&=\int f(x,y) dx,\\
    f_{\neq} &=f-f_=.
\end{align*}

\subsection*{The adapted unknowns }
For the following, it is useful to change to the unknowns $p_{1,\neq} =\Lambda_t^{-1}  \nabla_t^\perp v_{\neq} $ and $p_{2,\neq} =\Lambda_t^{-1}  \nabla_t^\perp b_{\neq}$. However, since $\Lambda_t^{-1}  \nabla_t^\perp$ is not a bounded operator on the $x$ average,  we define 
\begin{align*}
    p_{1,\neq } &= \Lambda^{-1}_t \nabla^\perp_t v_{\neq},\\
    p_{1,=}&= v_{1,=},\\
    p_{2,\neq } &= \Lambda^{-1}_t \nabla^\perp_t b_{\neq},\\
    p_{2,=}&= b_{1,=}. 
\end{align*}
Thus \eqref{eq:vb} can be equivalently expressed as 
\begin{align}
\begin{split}
     \partial_t p_1 - \partial_x \partial_y^t \Delta^{-1}_t p_1- \alpha \partial_x p_2 &= \nu  \Delta_t p_1 +\Lambda^{-1}_t  \nabla^\perp_t (b\nabla_t b- v\nabla_t v), \\
  \partial_t p_2 +\partial_x \partial_y^t \Delta^{-1}_t p_2 - \alpha \partial_x p_1 &= \kappa \Delta_t p_2  +\Lambda^{-1}_t \nabla^\perp_t (b\nabla_t v- v\nabla_t b),\\
  p|_{t=0} &= p_{in}\label{eq:p1}. 
\end{split}
\end{align} 
These unknowns are particularly useful since 
\begin{align*}
    \Vert A p_1 \Vert_{L^2} &= \Vert A v\Vert_{L^2},\\
    \Vert A p_2\Vert_{L^2} &= \Vert A b\Vert_{L^2},
\end{align*}
for all Fourier multipliers $A$ such that one side is finite.

\section{Linear Stability}\label{sec:linbeh}
In this section, we consider the behavior of the linearized version of \eqref{eq:p1}:
\begin{align}
\begin{split}
    \p_t p_1 - \partial_x \partial_y^t \Delta^{-1}_t p_1- \alpha \partial_x p_2 &= \nu \Delta_t p_1 , \\
    \p_t p_2 +\partial_x \partial_y^t \Delta^{-1}_t p_2 - \alpha \partial_x p_1 &= \kappa  \Delta_t p_2  .\label{eq:plin}
\end{split}
\end{align}
For this equation, we establish the following proposition: 
\begin{pro}\label{pro:lin}
    Consider $\alpha >\tfrac 1 2 $ and $0<\kappa\le \nu$. Let  $p_{in}\in H^N $ with $p_{in,=}=\int p_{in} dx=0$,  then the solution $p$ of \eqref{eq:plin} satisfies the bound
\begin{align}
    \Vert p(t)\Vert_{H^N}\lesssim e^{-c\kappa^{\frac 1 3 }t} (1+\kappa^{-\frac 1 3 } \nu)\Vert p_{in} \Vert _{H^N}.\label{eq:linup} 
\end{align}
\end{pro}
For the proof of Proposition \ref{pro:lin}, we perform a Fourier transform $(x,y)\mapsto (k,\xi)$ and  replace $p_1$ by $ip_1$ \eqref{eq:plin} to infer for modes $k \neq 0$
\begin{align}
\begin{split}
    \p_t p_1 &= -\tfrac {t-\frac \xi k } {1+(t-\frac \xi k)^2} p_1- \alpha kp_2 - \nu k^2(1+(t-\tfrac \xi k)^2) p_1 , \\
    \p_t p_2 &=\tfrac {t-\frac \xi k}  {1+(t-\frac \xi k)^2} p_2 + \alpha k p_1 - \kappa k^2(1+(t-\tfrac \xi k)^2) p_2  .\label{eq:plft}
\end{split}
\end{align}
Where with slight abuse of notation we omit writing the Fourier transformation. This equation has several effects that appear on different regimes of $t-\tfrac \xi k$, which we discuss in the following. The effect of circular movement appears on $\vert t-\tfrac \xi k \vert \lesssim \nu^{-1} $. 
We first sketch these effects 

\subsection*{Circular movement} \label{sec:cm}
To highlight the effect of the constant magnetic field $\alpha $ in \eqref{eq:plin} we consider the toy model
\begin{align}
\begin{split}
    \p_t p_1&= -\alpha k p_2 , \\
    \p_t p_2 &=\alpha k p_1.\label{eq:circ}
\end{split}
\end{align}
This is solved by 
\begin{align*}
    p(t)&= \left( 
    \begin{array}{cc}
         \cos(\alpha t )& -\sin(\alpha k t ) \\
         \sin(\alpha t )& \cos(\alpha kt )
    \end{array}
    \right) p_{in}. 
\end{align*}
We call this effect of the constant magnetic field \eqref{eq:circ} circular movement, which leads to a transfer between $p_1$ and $p_2$. This circular movement is counteracted by viscosity for times away from $\tfrac \xi k $. 

\subsection*{Effect of strong viscosity }\label{sec:svis}
Let us consider the case when $0< \kappa \ll \nu $ and for simplicity of notation let $k=1$ and $\xi =0$. Due to the viscosity, we obtain $p_1\approx 0 $ for large times $t\ge t_0 \gg 1$.  Then from  \eqref{eq:plft} we deduce the toy model
\begin{align}
\begin{split}
    \p_t p_2 &=(\tfrac {t }  {1+t^2}  - \kappa (1+t^2))  p_2.\label{eq:sv}
\end{split}
\end{align}
The first term in \eqref{eq:sv} leads to linear growth until the resistivity is strong enough for the second term to take over. This is seen in the explicit solution of \eqref{eq:sv} 
\begin{align*}
    p_2(t) &= \tfrac {\langle t\rangle}{\langle t_0  \rangle}\exp(- \kappa  \int^t_{t_0} 1+((\tau-t) )^2 \ d\tau ) p_2(t_0).
\end{align*}
This is estimated by 
\begin{align*}
    p_2(t)&\lesssim  t_0^{-1} \kappa^{-\frac 1 3 }  e^{-c \kappa^{\frac 1 3 } (t-t_0)}p_2(t_0),
\end{align*}
which corresponds to the maximal growth which we obtain. In the following, we will see that $t_0 \approx \nu^{-1} $ is the time after which viscosity dominates. The reader may expect that the enhanced dissipation timescale $\nu^{-\frac 13 } $ would be the relevant timescale, but the combination of circular movement and the viscosity gives enough decay for $p_2$ such that the linear growth gets suppressed until the time $\nu^{-1}$.

\subsection*{Proof of Proposition \ref{pro:lin}}\label{sec:lin}
\begin{proof}
For simplicity of notation, we introduce the new variable $s=t-\frac \xi k $ and initial time $s_{in}=-\tfrac \xi k $. Then equation \eqref{eq:plft} reads
\begin{align*}
    \p_s p_1 &= -\tfrac s {1+s^2} p_1- \alpha kp_2 - \nu k^2(1+s^2) p_1 , \\
    \p_s p_2 &=\tfrac s {1+s^2} p_2 + \alpha k p_1 - \kappa k^2(1+s^2) p_2.
\end{align*}
Further we change the unknown to  $\tilde p = \exp(-\tfrac \kappa 2 k^2 (s-s_{in}+\tfrac 1 3 (s^3-s_{in}^3)))p $. For $\tilde \kappa = \tfrac \kappa 2 $ and  $\tilde \nu = \nu - \tfrac \kappa 2 $, this yield the equation
\begin{align*}
    \p_s \tilde p_1 &= -\tfrac s {1+s^2} \tilde p_1- \alpha k\tilde p_2 - \tilde \nu k^2(1+s^2) \tilde p_1 , \\
    \p_s \tilde p_2 &=\tfrac s {1+s^2} \tilde p_2 + \alpha k \tilde p_1- \tilde \kappa k^2(1+s^2) \tilde p_1.
\end{align*}
Let us denote $s_0:=\nu^{-1}$ and in the following, we distinguish between times $\vert s\vert \le s_0 $ and $\vert s \vert \ge s_0$. We first consider the case $s_{in}\le -s_0$. For $\vert s \vert \le s_0$, the circular movement is not suppressed by the viscosity. 

We define the energy  $E=\vert \tilde p\vert^2+\tfrac 1 {\alpha k }\tfrac {2s}{1+s^2}\tilde  p_1\tilde p_2 $, then $E$ is a positive quadratic form due to our assumption $\alpha>\tfrac 1 2 $ and satisfies
\begin{align*}
    (1-\tfrac 1 {2\alpha k })  \vert \tilde   p\vert^2\le E\le (1+\tfrac 1 {2\alpha k }) \vert \tilde  p\vert^2.
\end{align*}
We calculate the time derivative
\begin{align*}
    \p_s E &+  \tilde \nu k^2(1+s^2)\tilde p_1^2+  \tilde \kappa  k^2(1+s^2)\tilde p_2^2 \\
    &= \tfrac 1{\alpha k }\p_s(\tfrac {2s}{1+s^2})  \tilde p_1 \tilde p_2- 2s\tfrac {(\tilde \nu -\tilde \kappa)k} {\alpha  }  \tilde p_1 \tilde p_2 \\
    &\le \tfrac 1{\alpha k }\p_s(\tfrac {2s}{1+s^2})  \tilde p_1 \tilde p_2+\tfrac 1 2 \tilde \nu k^2 (1+s^2) \tilde  p_1^2 +\tfrac {2\tilde \nu} {\alpha^2  } \tilde p_2^2
\end{align*}
and so with $\vert \tilde p\vert^2 \le \tfrac {2\alpha }{2\alpha -1 }  E$ we infer 
\begin{align*}
    \vert \p_s  E \vert  &\le \tfrac {\alpha } {\alpha -\frac 1 2 }( \tfrac 1{1+s^2} +2\tfrac {\tilde \nu} {\alpha^2}   ) E.
\end{align*}
Gronwall's lemma implies 
\begin{align*}
    E(s_0)&\le \exp \left(\tfrac {\alpha} {\alpha -\frac 1 2 }(\pi  +2 \tfrac {\tilde \nu} {\alpha^2}\vert s_0\vert   ) r\right)E(-s_0).
\end{align*}
Since $\nu  s_0= 1$, we deduce 
\begin{align*}
    E(s_0)&\lesssim E(-s_0)
\end{align*}
and thus 
\begin{align}
    \vert  \tilde p(s_0)\vert\lesssim \vert  \tilde p(- s_0)\vert .\label{eq:ps0}
\end{align}
Consider the case  $\vert s\vert \ge s_0$, we calculate 
\begin{align*}
    \tfrac 1 2 \p_s \vert \tilde p\vert^2  &\le (-\tfrac { s  }{1+s^2}- \tilde \nu k^2(1+s^2)) \tilde  p_1^2\\
    &+(\tfrac { s  }{1+s^2}- \tilde \kappa k^2(1+s^2)) \tilde  p_2^2,
\end{align*}
and since $(-\tfrac { s  }{1+s^2}- \tilde \nu k^2(1+s^2))\le 0$ for all $\vert s \vert \ge s_0$ we conclude  
\begin{align*}
    \p_s \vert \tilde p\vert^2 &\le(\tfrac { 2  s  }{1+s^2}- \tilde \kappa k^2(1+s^2))_+  \tilde p_2^2.
\end{align*}
Thus we obtain the estimate 
\begin{align*}
    \vert  \tilde p (s) \vert^2&\le \left\{ \begin{array}{cc}
         \vert  \tilde p(s_{in})\vert^2 & s\le -s_0,  \\
         \tfrac {1+s^2}{1+s_0^2} \vert \tilde  p(s_0)\vert^2&  s_0\le s\le 2 (\kappa k^2 )^{-\frac 13 },\\
         (1+ 4\nu^2\kappa^{-\frac 2 3 }k^{-\frac 4 3 })\vert \tilde  p(s_0)\vert^2 &  s_0\vee 2 (\kappa k^2 )^{-\frac 13 }\le s .
    \end{array}\right.
\end{align*}
Combining this with \eqref{eq:ps0} we infer 
\begin{align*}
    \vert \tilde p(s) \vert \lesssim (1+ \nu \kappa^{-\frac 1 3 }k^{-\frac 2 3 })\vert  p(s_{in})\vert .
\end{align*}
The case $s_{in} \ge -s_0$ is established similarly since we only bound the growth. With  
\begin{align*}
    \exp(-\tfrac \kappa 2 k^2 (s-s_{in}+\tfrac 1 3 (s^3-s_{in}^3)))\lesssim  e^{-c\kappa^{\frac 1 3 } t }
\end{align*}
we deduce 
\begin{align*}
    \vert  p(s) \vert 
    &\lesssim e^{-c\kappa^{\frac 1 3 }t} \vert \tilde  p \vert  (s) \\
    &\lesssim (1+ \nu\kappa^{-\frac 1 3 }k^{-\frac 2 3 }) e^{-c\kappa^{\frac 1 3 }t} \vert  p\vert (s_{in}).
\end{align*}
Equation \eqref{eq:plft} decouples in $\xi$ and $k$, so we infer the proposition with this estimate. 

\end{proof}

\section{Sobolev Stability for the Nonlinear System}\label{sec:StabThres}
The following theorem is a more general statement of Theorem \ref{thm:simp}.  We dedicate the remainder of the section to the proof.
\begin{thm}\label{thm:comp}
    Let $\alpha > \tfrac 1 2 $ and $N\ge 5$, then there exist $c_0,c>0$, such that for all $0<\kappa \le \nu \le \tfrac 1{ 40}  (1-\tfrac 1{2\alpha })^{\frac 65 } $ there exist $L =\max (1, \nu \kappa^{-\frac 1 3}) $,  such that for all initial data, which satisfy 
    \begin{align}
    \begin{split}
        \Vert (v, b)_{in, \neq}\Vert_{H^N } &= \eps \le c_0 L^{-1} \nu^{\frac 1 {12}} \kappa^{\frac 1 2 },\\ \label{eq:init}
        \Vert (v, b)_{in, =} \Vert_{H^N }&\le\tilde \eps, \qquad \qquad  \text{with }  \eps\le  \tilde \eps \le \nu^{-\frac 1 {12}}\eps 
    \end{split}
    \end{align}
    the corresponding solution of \eqref{eq:vb} satisfies the bound 
    \begin{align}
    \begin{split}
         \Vert (v, b)_{ \neq}(t)\Vert_{L^\infty H^N }+  \Vert\nabla_t  (\nu v,\kappa  b)_{ \neq}\Vert_{L^2H^N }&\lesssim L e^{-c\kappa^{\frac 1 3 } t } \eps, \\
         \Vert (v, b)_{ =}(t)\Vert_{L^\infty H^N }+  \Vert\p_y  (\nu v,\kappa  b)_{ =}\Vert_{L^2 H^N }&\lesssim \tilde \eps.\label{eq:vbthm}
    \end{split}
    \end{align}
 Furthermore, we obtain the following enhanced dissipation estimates 
    \begin{align*}
        \Vert v_{\neq}\Vert_{L^2 H^N} &\lesssim L \nu^{-\frac 1 6 } e^{-c\kappa^{\frac 1 3 } t } \eps,\\
        \Vert b_{\neq}\Vert_{L^2 H^N} &\lesssim L \kappa^{-\frac 1 6 }e^{-c\kappa^{\frac 1 3 } t }  \eps.
    \end{align*}
\end{thm}
This theorem implies Theorem \ref{thm:simp}. With slight abuse of notation, we write $L$ as the $\nu$ and $\kappa$ dependent part of the Lipschitz constant. We prove this theorem by using a bootstrap method. Let $A$ be the Fourier weight 
\begin{align*}
    A:&= M \vert \nabla \vert^N e^{c\kappa^{\frac 1  3} t \textbf{1}_{\neq} },
\end{align*}
where $M=M_L M_1 M_\kappa  M_{\nu} M_{\nu^3} $ are defined as
\begin{align*}    
    \tfrac {-\dot M_L } {M_L } &=    \tfrac {t-\frac \xi k  }{1+(\frac \xi k-t)^2} \textbf{1}_{\{\nu ^{-1 } \le t-\frac \xi k \le (c_1\kappa k^2) ^{-\frac 1 3 } \} }&&k\neq 0,\\
    \tfrac {-\dot M_1 } {M_1} &= C_\alpha  \tfrac { \vert k \vert+ \nu^{\frac1{12}}\vert k\vert^2  }{k^2 +(\xi-kt)^2}&&k\neq 0,\\
    \tfrac {-\dot M_\nu  } {M_\nu } &=\tfrac {\nu^{\frac 1 3 } }{1+ \nu^{\frac 2 3 } (t-\frac \xi k )^2} &&k\neq 0,\\ 
    \tfrac {-\dot M_\kappa } {M_{\kappa} } &= \tfrac {\kappa^{\frac 1 3 } }{1+ \kappa^{\frac 2 3 } (t-\frac \xi k )^2}&&k\neq 0,\\
    \tfrac {-\dot M_{\nu^3} } {M_{\nu^3}} &= \tfrac {C_\alpha  \nu}{1+ \nu^2(t-\frac \xi k )^2}&&k\neq 0,\\
    M_\cdot (t=0)&= M_\cdot (k=0)=1. 
\end{align*}
The weight $M_L$ is an adaption of the weight $m^{\frac 1 2}$ in \cite{liss2020sobolev} to our setting. The method of using time-dependent Fourier weights is common when working with solutions around Couette flow and the other weights are modifications of previously used weights (cf. \cite{bedrossian2016sobolev,masmoudi2022stability,liss2020sobolev,zhao2023asymptotic}).  For simplicity, here we only state their main properties and refer to Appendix \ref{app:weight} for a detailed description. The constants  $C_\alpha = \tfrac 2{\min(1, \alpha -\frac 1 2)  }$, $ c= \tfrac 1{200} (1-\tfrac 1{2\alpha }) ^2$ and $ c_1= \tfrac 1{20} (1-\tfrac 1{2\alpha })$ are determined through the linear estimates. For the weights we obtain
\begin{align}
\begin{split}
    L^{-1}\le \min(1, \nu^{-1 }\kappa^{\frac 1 3 }k^{\frac 2 3 } ) \lesssim M_L &\le 1,\\
    M_1\approx M_\kappa\approx M_{\nu} \approx M_{\nu^3} &\approx 1.\label{eq:Mapp}
\end{split}
\end{align}
We note that the weight $M_L $ is distinct from the others due to its lower bound $L^{-1}$, which depends on $\nu$ and $\kappa$. The weight $M_L$ is necessary to bound the linear growth in the region  $\nu ^{-1 } \lesssim t-\frac \xi k \lesssim (\kappa k^2) ^{-\frac 1 3 }$. Controlling the effects of $M_L$ is one of the main challenges in the proof of Theorem \ref{thm:comp}. We recall the unknowns $p$ and equation \eqref{eq:p1}
\begin{align}
\begin{split}
     \partial_t p_1 - \partial_x \partial_y^t \Delta^{-1}_t p_1- \alpha \partial_x p_2 &= \nu  \Delta_t p_1 +\Lambda^{-1}_t  \nabla^\perp_t (b\nabla_t b- v\nabla_t v), \\
  \partial_t p_2 +\partial_x \partial_y^t \Delta^{-1}_t p_2 - \alpha \partial_x p_1 &= \kappa \Delta_t p_2  +\Lambda^{-1}_t \nabla^\perp_t (b\nabla_t v- v\nabla_t b),\\
  p|_{t=0}&= p_{in}. \label{eq:p} 
\end{split}
\end{align}
Let $\chi\in C^\infty(\R_+\times \Z\times \R)  $ be a Fourier multiplier defined by
\begin{align}
    \chi&= \chi(k,\xi)=\left\{
    \begin{array}{cc}
         1 & \vert t-\frac \xi k \vert \le \   \nu^{-1 }  \\
         0 & \vert t-\frac \xi k \vert \ge 2 \nu^{-1 } 
    \end{array}
    \right.\label{eq:chi}\\
    \p_t \chi &\le 2 \nu\label{eq:pchi}.
\end{align}
We define the main energy 
\begin{align*}
    E:&= \Vert A p_{\neq}\Vert_{L^2}^2 + \tfrac 2\alpha \langle \p_y^t \Delta_t^{-1}  \chi  A p_{1,\neq} ,A p_{2,\neq} \rangle.
\end{align*}
As $\alpha >\tfrac 1 2 $, this energy is positive definite and satisfies
\begin{align}
    (1-\tfrac 1 {2\alpha} ) \Vert A p \Vert_{L^2}  \le E\le (1+\tfrac 1 {2\alpha }) \Vert A p \Vert_{L^2}.\label{eq:Ep}
\end{align}
 In the following, we assume initial data as in Theorem \ref{thm:comp}, i.e. \eqref{eq:init}.  We use a bootstrap approach to prove the following two estimates globally in time: \\
The \textbf{energy estimate without $x$-average }
\begin{align}
\begin{split}
    \Vert E\Vert_{L^\infty } &+\Vert  A\nabla_t \otimes (\nu p_{1,\neq} , \kappa p_{2,\neq}  )  \Vert_{L^2L^2}^2  \\
    &+ \sum_{j=1,\nu,\kappa,\nu^3} \Vert \sqrt{\tfrac {-\dot M_j } {M_j} } A  p_{\neq} \Vert_{L^2 L^2 }^2\le(C\eps)^2\label{eq:Eout}
\end{split}
\end{align}
The \textbf{energy estimate with $x$-average}
\begin{align}
\begin{split}
    \Vert  p_= \Vert_{L^\infty H^N}^2 &+\Vert  \p_y  (\nu p_{1,=}, \kappa p_{2,=} )  \Vert_{L^2H^N}^2  \le (C\tilde \eps) ^2. \label{eq:Ewith}
\end{split}
\end{align}
We then prove that the equality in the estimates is not attained at time $T$. By local wellposedness, the estimates thus remain valid at least for a short additional time. This contradicts the maximality and thus $T$ has to be infinite. We note that we suppress in our notation the $T$ in the estimates (see \eqref{eq:lebT}). With $1\le \kappa^{-\frac 1 3 } (\tfrac {-\dot M_\kappa }{M_\kappa} + \kappa k^2 (1+(t-\tfrac \xi k )^2 ) $ and $1\le \nu^{-\frac 1 3 } (\tfrac {-\dot M_{\nu} }{M_{\nu}} + \nu k^2 (1+(t-\tfrac \xi k )^2 ) $) we infer from \eqref{eq:Ep} and  \eqref{eq:Eout}  the enhanced dissipation estimates 
\begin{align}
    \Vert A p_{1,\neq }\Vert_{L^2 L^2}&\le 2(1-\tfrac 1 {2\alpha} )^{-1}  \nu^{-\frac 16 }    C\eps  \label{eq:enh1},\\
    \Vert A p_{2,\neq }\Vert_{L^2 L^2}&\le 2(1-\tfrac 1 {2\alpha} )^{-1}  \kappa^{-\frac 1 6} C\eps  \label{eq:enh2}.
\end{align}
By the construction of $M_1$ we obtain 
\begin{align*}
    \Vert \p_x \Lambda_t^{-1} A  p \Vert_{L^2} \lesssim \nu^{-\frac 1 {12} } \eps.
\end{align*}
We obtain the energy estimate by deriving the energy $E$ 
\begin{align}
\begin{split}
    \p_t E &+2 \Vert  A\nabla_t \otimes (\nu p_{1,\neq}, \kappa p_{2,\neq} )  \Vert_{L^2}^2  + 2 \Vert \sqrt{\tfrac {-\dot M } M } A  p_{\neq}\Vert_{L^2  }^2\\
    =& 2c\kappa^{\frac 1 3 } \Vert Ap_{\neq} \Vert^2_{L^2} \\
    &-2\langle A(1-\chi) p_{1,\neq} , \partial_x \partial_x^t \Delta^{-1}_t Ap_{1,\neq} \rangle \\
    &+2\langle A(1-\chi)p_{2,\neq} , \partial_x \partial_x^t \Delta^{-1}_t Ap_{2,\neq}  \rangle\\
    &+\tfrac 4 \alpha  \langle \chi \p_y^t\Delta_t^{-1} Ap_{1,\neq}  , \dot A p_{2,\neq}\rangle \\
    &+\tfrac 2 \alpha\langle \chi  \p_t(\p_y^t \Delta_t^{-1} ) Ap_{1,\neq}  , Ap_{2,\neq} \rangle \\
    &+\tfrac 2 \alpha \langle \p_t(\chi) \p_y^t \Delta_t^{-1}  Ap_{1,\neq}  , Ap_{2,\neq} \rangle \\
    &+ \tfrac {2 \vert \nu +\kappa\vert}\alpha \langle \chi \p_y^t  A p_{1,\neq} , Ap_{2,\neq}\rangle \\
    &+ 2 \langle A  v_{\neq} , A(b\nabla_t b- v\nabla_t v)\rangle  \\
    &+ 2 \langle A b_{\neq} , A(b\nabla_t v- v\nabla_t b) \rangle \\
    &+ \tfrac 2 \alpha\langle  \chi A \p_y^t \Delta^{-1}_t b_{\neq},   A   (b \nabla_t b - v\nabla_t v )\rangle  \\
    &+ \tfrac 2 \alpha\langle \chi A \p_y^t \Delta^{-1}_t   v_{\neq},   A  (b\nabla_t v- v\nabla_t b)   \rangle \\   
    =& L_1+ L_{NR}+ L_R +NL_{\neq}  +ONL. \label{eq:energy}
\end{split}
\end{align}
Where we denoted by $ONL$ all the terms which include the operator $\p_y^t \Delta_t^{-1}$ and $NL$ the one which does not. 
Furthermore, for the energy of $x$-averages, we obtain
\begin{align}
\begin{split}
    \p_t \Vert  & p_=\Vert_{H^N}+\Vert  \p_y (\nu p_{1,=},\kappa p_{2,=})\Vert_{H^N }\\
    &\le \langle \langle \p_y\rangle^N v_{1,=}, \langle \p_y\rangle^N(b\nabla_t b- v\nabla_t v)_=\rangle \\
    &\quad +  \langle \langle \p_y\rangle^Nb_{1,=}, \langle \p_y\rangle^N(b\nabla_t v- b\nabla_t v)_=\rangle \\
    &=NL_=.\label{eq:energy_av}
\end{split}
\end{align}
In the following subsections, we establish the following proposition:
\begin{prop}[Control of errors]
\label{prop:errors}
  Under the assumptions of Theorem \ref{thm:comp}, there exists a constant  $C=C(\alpha)>0$ such that if \eqref{eq:Eout} and \eqref{eq:Ewith} are satisfied for $T>0$, then the following estimate holds
 \begin{align}
 \begin{split}
   \int_0^T L_1+ L_R +L_{NR} \ dt &\le \tfrac {17+ \frac 3 2 \alpha }{10}(C \eps)^2+2\Vert \sqrt {\tfrac {-\dot M_L}{M_L }} Ap_2\Vert_{L^2L^2}^2 ,\\
   \int_0^T  NL_{\neq}+ ONL \ dt & \lesssim  L \nu^{-\frac 1 {12} } \kappa^{-\frac 1 2 }  \eps^3+ (L\kappa^{-\frac 1 3 }+\kappa^{-\frac 12}) \tilde\eps \eps^2,\\
   \int_0^T  NL_= \  dt & \lesssim  L \nu^{-\frac 1 {12} } \kappa^{-\frac 1 2 }  \tilde \eps \eps^2 .     \label{eq:coe}
 \end{split}
 \end{align}
\end{prop}
With this  proposition we deduce Theorem \ref{thm:comp}:
\begin{proof}[Proof of Theorem \ref{thm:comp}] 
By a standard application of the  Banach fixed-point theorem we obtain local well-posedness, see Appendix \ref{sec:ex}. Thus for all initial data, there exists a time interval $[0,T]$ such that \eqref{eq:Eout} and \eqref{eq:Ewith} hold. Let $T^\ast$ be the maximal time such that  \eqref{eq:Eout} and \eqref{eq:Ewith} hold. Let $c_0$ be a given, small constant and suppose for the sake of contradiction that $T^\ast<\infty$.
With the estimates \eqref{eq:energy}, \eqref{eq:energy_av} and \eqref{eq:coe} and since $c_0$ is small we obtain that the estimates \eqref{eq:Eout} and \eqref{eq:Ewith} do not attain equality. Thus by local existence, $T^\ast$ is not the maximal time and thus we obtain a contradiction. Therefore, for small enough $c_0$, \eqref{eq:Eout} and \eqref{eq:Ewith} hold global in time and so we infer Theorem \ref{thm:comp}.
\end{proof}
The remainder of the section is dedicated to the proof of Proposition \ref{prop:errors}. We rearrange and use partial integration to infer that
\begin{align}
\begin{split}
    &\quad \langle A  v_{\neq} , b\nabla_t A b_{\neq}- v\nabla_t A v_{\neq} \rangle + \langle A b_{\neq} , b\nabla_t A v_{\neq}- v\nabla_t A b_{\neq} \rangle\\
    &=\langle  b, \nabla_t(A  v_{\neq} A b_{\neq})\rangle - \tfrac 1 2 \langle  v, \nabla_t(A  v_{\neq} A v_{\neq})+ \nabla_t(A  b_{\neq} A b_{\neq})\rangle \\
    &=0 \label{eq:rearangeNL}.
\end{split}
\end{align}
The $NL$ term consists of trilinear products with the unknowns 
\begin{align}
 a^1a^2a^3\in \{vvv,vbb,bbv,bvb\}   \label{eq:a}.
\end{align}
Thus, we denote the nonlinear terms 
\begin{align*}
    NL_{\neq}[a^1a^2a^3]=&\langle A a^1_{\neq} , A(a^2_{\neq}\nabla_t a^3_{\neq})_{\neq} -a^2_{\neq}\nabla_t A a^3_{\neq} \rangle,\\
    & +\langle A a^1_{\neq} , A(a^2_=\nabla_t a^3_{\neq})-a^2_=\nabla_t a^3_{\neq} \rangle,\\
    &+ \langle A a^1_{\neq} , A(a^2_{\neq}\nabla_t a^3_=) \rangle,\\
     NL_{=}[a^1a^2a^3]=&\langle \langle \p_y\rangle^N  a^1_{=} , \langle \p_y\rangle^N(a^2_{\neq}\nabla_t a^3_{\neq})_=\rangle.
\end{align*}
If we do not use specific choices for $a^1a^2a^3$ we write just $NL$. Similarly, we use  $a^1a^2a^3\in \{bvv,bbb,vbv,vvb\}$ for $ONL$. Furthermore, we always use $i$ such that $p_i = \Lambda_t^{-1}\nabla^\perp_t  a^2$ in the sense that $i=1$ if $a^2=v$ and $i=2$ if $a^2=b$. 
We perform the energy estimates in the next subsections: 
\begin{itemize}
    \item In Subsection \ref{sec:linest} we estimate the linear error terms. In this subsection, the split with $\chi$ into resonant and non-resonant regions depending on $\nu$ is vital. 
    \item In Subsection \ref{sec:no} we conclude the energy estimate for the nonlinear term without $x$ average. Here it is necessary to perform a low and high frequency decomposition. This gives us a reaction and a transport term. In particular, for $\kappa\downarrow 0$ bounding the transport term is very challenging due to the linear growth.     
    \item In Subsections \ref{sec:second}, \ref{sec:third} and \ref{sec:first} we estimate nonlinear terms with an $x$-average component.    
    \item In Subsection \ref{sec:ONL} we estimate nonlinear term which arise due $\chi$ in the resonant regions.  For these terms, we obtain an additional $\Lambda_t^{-1}$, which has a stabilizing effect. This stabilizing effect is necessary due to a nonlinear term consisting of only magnetic components. 
\end{itemize}

\subsection{Linear estimates }\label{sec:linest}
In this section, we establish estimates of the linear errors $L_1$, $L_R$ and $L_{NR}$ of \eqref{eq:energy}. In order to estimate $L_1$, we use \eqref{eq:enh1} and \eqref{eq:enh2} to deduce
\begin{align*}
    \int L_1 d\tau &=2c\kappa^{\frac 1 3} \Vert Ap_{\neq }\Vert^2_{L^2L^2}\le 8(1-\tfrac 1 {2\alpha} )^{-1} c(C \eps)^2. 
\end{align*}
 For the $L_{NR }$ terms in \eqref{eq:energy},  we infer
\begin{align*}
     \langle A (1-\chi)  p_{1,\neq} , A\p_x \p_x^t \Delta^{-1}_t p_{1,\neq}\rangle  &=\sum_{k\neq 0} \int d\xi (1-\chi) \tfrac {t-\frac \xi k }{1+(t-\frac \xi k )^2}A^2 p_{1}^2 \\
    &\le \nu^3 \Vert A\nabla_t p_{1,\neq } \Vert^2_{L^2},
\end{align*}
since  $(1-\chi)\tfrac {t-\frac \xi k }{1+(t-\frac \xi k )^2}\le (1-\chi)(\nu )^3(1+(t-\tfrac \xi k )^2) $ due to $\chi=1$ for $\vert t-\frac \xi k\vert \le \nu^{-1}$. Furthermore, using \eqref{eq:M2est} we estimate
\begin{align*}
    \langle A(1-\chi) p_{2,\neq} , A\partial_x \partial_x^t \Delta^{-1}_t p_{2}\rangle &=\sum_{k\neq 0} \int d\xi(1-\chi)  \tfrac {t-\frac \xi k }{1+(t-\frac \xi k )^2}A^2p_2^2 \\
    &\le  \sum_{k\neq 0} \int d\xi (1-\chi)(\tfrac {-\dot M_L}{M_L }+\kappa  c_1 (1+(t-\tfrac \xi k )^2))A^2p_2^2\\
    &\le\Vert  \sqrt {\tfrac {-\dot M_L}{M_L }} Ap_{2,\neq}\Vert_{L^2}^2 + \kappa c_1 \Vert  \nabla_t  Ap_{2,\neq}\Vert_{L^2}^2.
\end{align*}
 Thus with \eqref{eq:Eout} we deduce 
\begin{align*}
    \int  L_{NR} \  d\tau &\le (2 \nu^2  + 2 c_1  )(C \eps)^2 + 2\Vert \sqrt {\tfrac {-\dot M_L}{M_L }} Ap_{2,\neq}\Vert_{L^2L^2}^2.
\end{align*}
For $L_{R }$, we estimate in frequency space
\begin{align*}
   \vert  (\textbf{1}_{\neq}\p_y^t\Delta_t^{-1})^\wedge  \vert &= \vert (\tfrac {\xi-kt}{k^2+(\xi-kt)^2})_{k\neq 0}\vert \le \tfrac 1 2, \\
    \tfrac {-\dot M_L } {M_L } \vert (\p_y^t \Delta_t)^\wedge \vert &\le 
    \left\vert \left(\tfrac {(t-\frac \xi k)^2 }{k(1+(t-\frac \xi k)^2)^2}\right)_{k\neq 0}\right \vert \le C_\alpha^{-1}    \tfrac {-\dot M_1 } {M_1}.
\end{align*}
So it follows that
\begin{align*}
    &\tfrac 4 \alpha  \langle \chi \p_y^t\Delta_t^{-1} Ap_{1,\neq} , \dot A p_{2,\neq}\rangle \\
     &= \tfrac 4 \alpha  \langle \chi A p_{1,\neq},(c\kappa^{\frac 1 3 } + \tfrac {\dot M_1}{M_1 } +\tfrac {\dot M_L } {M_L } + \tfrac {\dot M_\kappa}{M_\kappa }+ \tfrac {\dot M_{\nu}}{M_{\nu} }+\tfrac {\dot M_{\nu^3}}{M_{\nu^3} })\partial_y^t \Delta^{-1}_t A  p_{2,\neq }\rangle\\
     &\le\tfrac {2c}\alpha \kappa^{\frac 1 3 }  \Vert A p_{\neq} \Vert^2_{L^2}+ (1+C_\alpha ^{-1} )\tfrac 1{\alpha} \Vert\sqrt{\tfrac {-\dot M_1}{M_1 }} A p_{\neq}\Vert_{L^2}^2 \\
     &\quad + \tfrac 1 {\alpha} \Vert  \sqrt{\tfrac {-\dot M_\kappa}{M_\kappa }}  A p_{\neq}\Vert_{L^2}^2+ \tfrac 1 {\alpha} \Vert  \sqrt{\tfrac {-\dot M_{\nu}}{M_{\nu} }} A p_{\neq}\Vert_{L^2}^2+ \tfrac 1 {\alpha} \Vert  \sqrt{\tfrac {-\dot M_{\nu^3}}{M_{\nu^3}}} A p_{\neq}\Vert_{L^2}^2.
\end{align*}
We use  the estimate in frequency space 
\begin{align*}
    \vert (((\p_x^2-(\p_y^t)^2) \Delta^{-2}_t )_{\neq})^\wedge\vert &=\left \vert  \left(\tfrac {1-(t-\frac \xi k )^2 }{k^2 (1+(t-\frac \xi k )^2 )^2 }\right)_{k\neq 0 } \right \vert \le C_\alpha^{-1}  \tfrac {-\dot M_1 }{M_1 },
\end{align*}
to infer that
\begin{align*}
    \tfrac 1 \alpha \langle \chi A p_{1,\neq} ,  A \p_x^{-1} (\p_x^2-(\p_y^t)^2) \Delta^{-2}_t p_{2,\neq}\rangle&\le C_\alpha^{-1}\tfrac 1 {\alpha} \Vert  \sqrt{\tfrac {-\dot M_1 }{M_1 }} A p_{1,\neq} \Vert_{L^2}\Vert  \sqrt{\tfrac {-\dot M_1 }{M_1 }} A p_{2,\neq} \Vert_{L^2}\\
    &\le C_\alpha^{-1}\tfrac 1 {2\alpha} \Vert  \sqrt{\tfrac {-\dot M_1 }{M_1 }} A p_{\neq} \Vert_{L^2}^2.
\end{align*}
With  \eqref{eq:pchi} we deduce 
\begin{align*}
    \langle \p_y^t \Delta_t^{-1} \p_t(\chi) A p_{1,\neq} , A p_{2,\neq}\rangle &\le \nu  \Vert A  p_{1,\neq} \Vert_{L^2} \Vert A\Lambda_t^{-1}  p_{2,\neq } \Vert_{L^2}.
\end{align*}
By the Fourier support of $\chi$  (see \eqref{eq:chi}) and the definition of $M_{\nu^3}$ we obtain $\chi \le   2 C_\alpha^{-1}  \nu^{-1} \tfrac {-\dot M_{\nu^3}}{M_{\nu^3}} \chi $, which yields 
\begin{align*}
    \tfrac {\vert \nu +\kappa\vert}\alpha \langle \chi A\p_y^t p_{1,\neq} , A p_{2,\neq}\rangle &\le 2  C_\alpha^{-1} \tfrac {\nu^{\frac 1 2 } } {\alpha }  \Vert A \p_y^t  p_{1,\neq} \Vert_{L^2} \Vert \sqrt{\tfrac {-M_{\nu^3}}{M_{\nu^3}}} A  p_{2,\neq }\Vert_{L^2} .
\end{align*}
Thus for the linear error $L_R$ we infer
\begin{align*}
    \int L_R d\tau  &\le ( \tfrac {2c}\alpha + \nu^{\frac 56 } ) (C \eps)^2  \\
     &\quad + \tfrac {1+ C_\alpha^{-1} }{\alpha} \Vert\sqrt{\tfrac {-\dot M_1}{M_1 }} p_{\neq}\Vert_{L^2L^2}^2 + \tfrac 1 {\alpha} \Vert \sqrt{\tfrac {-\dot M_\kappa}{M_\kappa }} p_{\neq}\Vert_{L^2L^2}^2+ \tfrac 1 {\alpha} \Vert  \sqrt{\tfrac {-\dot M_{\nu}}{M_{\nu} }} p_{\neq}\Vert_{L^2L^2}^2\\
      &\quad +\tfrac {1+ C_\alpha^{-1} }{2\alpha }  \Vert \sqrt{\tfrac {-\dot M_{\nu^3}}{M_{\nu^3}}} A  p_{2,\neq} \Vert_{L^2L^2} +C_\alpha^{-1} \tfrac {1 } {2\alpha }  \nu \Vert A \p_y^t  p_{1,\neq} \Vert_{L^2L^2}^2.
\end{align*}
Combining the estimates for all linear terms, we obtain
\begin{align*}
    \int L &+L_R +L_{NR} \ d\tau\\ &\le ((8+\tfrac 2\alpha )(1-\tfrac 1 {2\alpha} )^{-1} c+ 2 \nu^2  +2c_1+ \nu^{\frac 5 6 } )(C \eps)^2 \\
    &\quad + 2\Vert \sqrt {\tfrac {-\dot M_L}{M_L }} Ap_2\Vert_{L^2L^2}^2 \\
    &\quad +(1+\tfrac 3 2 C_\alpha^{-1})\tfrac 1{\alpha} \Vert\chi\sqrt{\tfrac {-\dot M_1}{M_1 }} p_{\neq}\Vert_{L^2L^2}^2 + \tfrac 1 {\alpha} \Vert \sqrt{\tfrac {-\dot M_\kappa}{M_\kappa }} p_{\neq}\Vert_{L^2L^2}^2+ \tfrac 1 {\alpha} \Vert  \sqrt{\tfrac {-\dot M_{\nu}}{M_{\nu} }} p_{\neq}\Vert_{L^2L^2}^2\\
      &\quad +\tfrac {1+ C_\alpha^{-1} }{2\alpha }  \Vert \sqrt{\tfrac {-M_{\nu^3}}{M_{\nu^3}}} A  p_{2,\neq} \Vert_{L^2L^2} +C_\alpha^{-1} \tfrac {1 } {2 \alpha }  \nu \Vert A \p_y^t  p_{1,\neq} \Vert_{L^2L^2}^2\\
      &\le (12 (1-\tfrac 1 {2\alpha} )^{-1} c+ 2 \nu^2  +2c_1 + \nu^{\frac 5 6 } + \tfrac {1+2C_\alpha^{-1}}{2\alpha} )(C \eps)^2\\
      &\quad +2\Vert \sqrt {\tfrac {-\dot M_L}{M_L }} Ap_{2,\neq }\Vert_{L^2L^2}^2 .
\end{align*}
Since  $\alpha > \tfrac 1 2 $ we deduce   $\tfrac {1+2C_\alpha^{-1}}\alpha <1+\tfrac 1 {2\alpha}$. Choosing the constants such that
\begin{align*}
    c&= \tfrac 1{200} (1-\tfrac 1{2\alpha })^2,\\
    c_1&= \tfrac 1{20} (1-\tfrac 1{2\alpha }),
\end{align*}
and recalling that
\begin{align*}
    \nu \le\tfrac 1{ 40}  (1-\tfrac 1{2\alpha })^{\frac 65 },
\end{align*}
we conclude that $(12 (1-\tfrac 1 {2\alpha} )^{-1} c+ 2 \nu^2 +2c_1 + \nu^{\frac 5 6 } + \tfrac {1+2C_\alpha^{-1}}{2\alpha} <\tfrac {17+\frac 3{2\alpha } }{20} $. Thus we obtain the estimate
\begin{align*}
    \int L_1 +L_R +L_{NR} d\tau       &\le \tfrac {17+ \frac 3 {2 \alpha} }{20}(C \eps)^2+2\Vert \sqrt {\tfrac {-\dot M_L}{M_L }} Ap_{2,\neq }\Vert_{L^2L^2}^2.
\end{align*}
This yields the first estimate of Proposition \ref{prop:errors}.

\subsection{Nonlinear terms without an $x$-average}\label{sec:no}

We apply the notation of \eqref{eq:a} and aim to estimate terms of the form
\begin{align*}
     &\quad \langle A a^1_{\neq} , A( a^2_{\neq}\nabla_t a^3_{\neq}) - a^2_{\neq} \nabla_t A a^3_{\neq}  \rangle\\
     &= \sum_{k,l,k-l \neq 0 }\iint d(\xi,\eta) \tfrac {A(k,\xi)- A(l,\eta )}{A(k-l, \xi-\eta) A(l, \eta ) }\tfrac {\xi l - k\eta} {((k-l)^2+(\xi-\eta-(k-l)t )^2 )^{\frac 12 }}\\
     &\qquad \qquad  (Aa^1)(k,\xi )(Ap_i)(k-l, \xi-\eta )(A a^3) (l,\eta)\\
     &=T+R.
\end{align*}
Here, we split the integral into the \emph{reaction} $R$ and the \emph{transport} $T$ terms which correspond to the sets 
\begin{align*}
    \Omega_R &= \{ \vert k-l,\xi-\eta \vert \ge \tfrac 1 8\vert l,\eta \vert\},\\
    \Omega_T &= \{ \vert k-l,\xi-\eta \vert <\tfrac 1 8\vert l,\eta \vert   \},
\end{align*}
in Fourier space. We split the weights 
\begin{align*}
    A(k,\xi ) - A(l,\eta )& = e^{ct\kappa^{\frac 1 3 }}(M_L (k,\xi) - M_L(l,\eta)) M_1(k,\xi)M_\kappa(k,\xi)M_{\nu}(k,\xi)M_{\nu^3}(k,\xi) \vert k,\xi \vert^N \\
    &\quad + e^{ct\kappa^{\frac 1 3 }}(\vert k,\xi \vert^N  - \vert l,\eta \vert^N ) M_L(l,\eta)M_1(k,\xi)M_\kappa(k,\xi )M_{\nu}(k,\xi)M_{\nu^3}(k,\xi )\\
    &\quad + e^{ct\kappa^{\frac 1 3 }}(M_1 (k,\xi) - M_1(l,\eta))  M_L(l,\eta)M_\kappa (k,\xi) M_{\nu}(k,\xi)M_{\nu^3}(k,\xi)\vert l,\eta \vert^N \\
    &\quad + e^{ct\kappa^{\frac 1 3 }}(M_\kappa (k,\xi) - M_\kappa(l,\eta)) M_1(l,\eta)M_L(l,\eta)M_{\nu}(k,\xi)M_{\nu^3}(k,\xi)\vert l,\eta \vert^N\\
    &\quad + e^{ct\kappa^{\frac 1 3 }}(M_{\nu} (k,\xi) - M_{\nu}(l,\eta)) M_1(l,\eta)M_L(l,\eta)M_\kappa(l,\eta) M_{\nu^3}(k,\xi) \vert l,\eta \vert^N\\
    &\quad + e^{ct\kappa^{\frac 1 3 }}(M_{\nu^3} (k,\xi) - M_{\nu^3}(l,\eta)) M_1(l,\eta)M_L(l,\eta)M_\kappa(l,\eta) M_{\nu}(l,\eta ) \vert l,\eta \vert^N
\end{align*}
and thus by \eqref{eq:Mapp} we estimate
\begin{align}
\begin{split}
    \tfrac {A(k,\xi)- A(l,\eta )}{A(k-l, \xi-\eta) A(l, \eta )}&\lesssim e^{-ct\kappa^{\frac 1 3 }} \tfrac {\vert M_L (k,\xi) - M_L(l,\eta)\vert }{M_L(k-l, \xi-\eta) M_L (l, \eta )}\tfrac {\vert\xi,\eta  \vert^N}{\vert l,\eta \vert^N\vert k-l,\xi -\eta  \vert^N}\\
    &\quad + e^{-ct\kappa^{\frac 1 3 }} \tfrac {\left\vert\vert k,\xi\vert^N - \vert l,\eta\vert^N\right\vert }{\vert l,\eta \vert^N\vert k-l,\xi -\eta  \vert^N}\tfrac 1 {M_L(k-l, \xi-\eta)}\\
    &\quad +e^{-ct\kappa^{\frac 1 3 }}\sum_{j=1,\kappa,\nu ,\nu^3}\vert M_j(k,\xi )- M_j(l,\eta )\vert \tfrac 1{\vert k-l,\xi-\eta  \vert^N}  \tfrac 1 {M_L(k-l, \xi-\eta)}.\label{eq:Asplit}
\end{split}
\end{align}
\textbf{Reaction term:} On the set $\Omega_R$ it holds that $\vert k-l,\xi-\eta \vert \ge\tfrac 1 8\vert l,\eta \vert $, thus $\vert k , \xi \vert , \vert l, \eta \vert \lesssim \vert  k-l , \xi-\eta \vert $. From \eqref{eq:Mapp}, \eqref{eq:Asplit} and $\tfrac {\vert k,\xi\vert^N }{\vert l,\eta \vert^N\vert k-l,\xi -\eta  \vert^N},\tfrac {\vert  \vert k,\xi\vert^N - \vert l,\eta\vert^N\vert }{\vert l,\eta \vert^N\vert k-l,\xi -\eta  \vert^N}\lesssim\tfrac {1}{\vert l,\eta \vert^N}$ we infer 
\begin{align*}
    \tfrac {A(k,\xi)- A(l,\eta )}{A(k-l, \xi-\eta) A(l, \eta )}&\lesssim \tfrac {1}{M_L(k-l, \xi-\eta) M_L (l, \eta )}\tfrac 1{\vert l,\eta \vert^N}.
\end{align*}
With $\xi l - k\eta= (\xi-\eta -(k-l )t) l - (k-l)(\eta-lt )$ and Hölder's inequality we deduce 
\begin{align*}
     R&= e^{-ct\kappa^{\frac 1 3 }}\sum_{k,l,k-l \neq 0 }\int d(\xi,\eta )\textbf{1}_{\Omega_R } \tfrac {A(k,\xi)- A(l,\eta )}{A(k-l, \xi-\eta) A(l, \eta ) } \tfrac {\xi l - k\eta} {((k-l)^2+(\xi-\eta-(k-l)\tau )^2 )^{\frac 12 }} \\ 
     &\qquad \qquad \qquad (Aa^1)(k,\xi )(Ap_i)(k-l, \xi-\eta )(A a^3) (l,\eta) \\
     &\lesssim   e^{-ct\kappa^{\frac 1 3 }}\sum_{k,l,k-l \neq 0 }\int d(\xi,\eta )\textbf{1}_{\Omega_R } \tfrac 1 {\vert l,\eta\vert^N }\tfrac 1{M_L(k-l, \xi-\eta) M_L (l, \eta ) } \tfrac {\vert( \xi-\eta -(k-l)t) l - (k-l)(\eta-lt )\vert } {((k-l)^2+(\xi-\eta-(k-l)\tau )^2 )^{\frac 12 }} \\ 
     &\qquad \qquad \qquad (Aa^1)(k,\xi ) (Ap_i)(k-l, \xi-\eta )(A a^3) (l,\eta) \\
     &\lesssim \Vert Aa^1_{\neq}\Vert_{L^2} \Vert \tfrac 1 {M_L} Ap_{i,\neq} \Vert_{L^2} \Vert \tfrac 1 {M_L}Aa^3_{\neq}\Vert_{L^2}\\
     &\quad +\Vert Aa^1_{\neq}\Vert_{L^2} \Vert \p_x\tfrac 1 {M_L}\Lambda^{-1}_t  Ap_{i,\neq} \Vert_{L^2} \Vert \tfrac 1  {M_L}A  \p_y^t a^3_{\neq}\Vert_{L^2}.
\end{align*}
We use \eqref{eq:M2inv1} and \eqref{eq:M2inv2} to infer 
\begin{align*}
R    &\lesssim \Vert Aa^1_{\neq}\Vert_{L^2} (\Vert Ap_{i,\neq} \Vert_{L^2}+ \nu \Vert (\Lambda_t \wedge \kappa^{-\frac 1 3}) Ap_{i,\neq} \Vert_{L^2}) (\Vert Aa^3_{\neq}\Vert_{L^2}+ \nu\Vert A (\Lambda_t \wedge \kappa^{-\frac 1 3}) a^3_{\neq}\Vert_{L^2})\\
     &\quad + L\Vert Aa^1_{\neq}\Vert_{L^2} (\Vert A \p_x \Lambda^{-1}_t p_{i,\neq} \Vert_{L^2}+ \nu\Vert  Ap_{i,\neq} \Vert_{L^2})  \Vert \p_y^t  A a^3_{\neq}\Vert_{L^2}.
\end{align*}
Integrating in time yields 
\begin{align*}
    \int R d\tau &\lesssim L \nu^{-\frac 1 {12}} \kappa^{-\frac 1 2 }\eps^3 .
\end{align*}
\textbf{Transport term:} On the set $\Omega_T$ it holds that $\vert k-l,\xi-\eta \vert <\tfrac 1 8\vert l,\eta \vert $ and thus it follows that $\vert k,\xi\vert\approx \vert l,\eta \vert $. By the mean value theorem, there exists  $\theta \in [0,1] $ such that
\begin{align*}
    \left\vert \vert k,\xi\vert^N - \vert l,\eta\vert^N \right \vert &\le N \vert k-l,\xi-\eta \vert \vert k- \theta l, \xi-\theta \eta \vert^{N-1} \\
    &\lesssim \vert k-l,\xi-\eta \vert \vert l,\eta \vert^{N-1}.
\end{align*}
Thus with  \eqref{eq:Asplit} and Lemma \ref{lem:Mdiff} we conclude, that 
\begin{align}
    \tfrac {A(k,\xi)- A(l,\eta )}{A(k-l, \xi-\eta) A(l, \eta )}&\lesssim\tfrac 1 {M_L(l,\eta )}  (\tfrac 1 {\vert l\vert }+\nu^{\frac 1 {12} })\tfrac 1 {\vert k-l, \xi-\eta \vert^{N-1}} \label{eq:nuT1}\\
    &\quad +\sum_{j=\kappa,\nu ,\nu^3} \vert M_j(k,\xi)- M_j(l,\eta)\vert \tfrac 1 {M_L(l,\eta )} \tfrac 1 {\vert k-l, \xi-\eta\vert^N}\label{eq:nuT1j} \\
    &\quad + \tfrac {M_L(k,\xi)- M_L(l,\eta )}{M_L(k-l, \xi-\eta)M_L(l, \eta ) } \tfrac 1 {\vert k-l, \xi-\eta \vert^N }\label{eq:nuT2}.
\end{align}
Based on this estimate, in the following we distinguish between different regimes in frequency,
\begin{align*}
     T
     &= \sum_{k,l,k-l \neq 0 }\int d(\xi,\eta )\textbf{1}_{\Omega_T} \tfrac {A(k,\xi)- A(l,\eta )}{A(k-l, \xi-\eta) A(l, \eta ) }\tfrac {\xi l - k\eta} {((k-l)^2+(\xi-\eta-(k-l)\tau )^2 )^{\frac 12 }}\\
     &\qquad \qquad (Aa^1)(k,\xi )(Ap_i)(k-l, \xi-\eta )(A a^3) (l,\eta) \\
     &\lesssim \sum_{k,l,k-l \neq 0 }\int d(\xi,\eta )\textbf{1}_{\Omega_T} \tfrac 1 {M_L(l,\eta )}  (\tfrac 1 {\vert l\vert }+\nu^{\frac 1 {12}} )\tfrac 1 {\vert k-l, \xi-\eta \vert^{N-1}}\tfrac {\xi l - k\eta} {((k-l)^2+(\xi-\eta-(k-l)\tau )^2 )^{\frac 12 }}\\
     &\qquad \qquad (Aa^1)(k,\xi ) (Ap_i)(k-l, \xi-\eta )(A a^3) (l,\eta)\\
     &\quad + \sum_{k,l,k-l \neq 0 }\int d(\xi,\eta )\textbf{1}_{\Omega_T} \textbf{1}_{\vert \frac \eta l -t \vert \ge\vert \frac {\xi-\eta}{k-l}-t \vert}\tfrac{\sum_{j=\kappa,\nu ,\nu^3}\vert M_j(k,\xi)- M_j(l,\eta)\vert} {M_L(l,\eta )} \tfrac 1 {\vert k-l, \xi-\eta \vert^{N}}\tfrac {\xi l - k\eta} {((k-l)^2+(\xi-\eta-(k-l)\tau )^2 )^{\frac 12 }}\\
     &\qquad \qquad  (Aa^1)(k,\xi ) (Ap_i)(k-l, \xi-\eta )(A a^3) (l,\eta)\\
     &\quad +\sum_{k,l,k-l \neq 0 }\int d(\xi,\eta )\textbf{1}_{\Omega_T} \textbf{1}_{\vert \frac \eta l -t \vert \le\vert \frac {\xi-\eta}{k-l}-t \vert}\tfrac{\sum_{j=\kappa,\nu ,\nu^3} \vert M_j(k,\xi)- M_j(l,\eta)\vert} {M_L(l,\eta )} \tfrac 1 {\vert k-l, \xi-\eta \vert^{N}}\tfrac {\xi l - k\eta} {((k-l)^2+(\xi-\eta-(k-l)\tau )^2 )^{\frac 12 }}\\
     &\qquad \qquad (Aa^1)(k,\xi ) (Ap_i)(k-l, \xi-\eta )(A a^3) (l,\eta)\\
     &\quad + \sum_{k,l,k-l \neq 0 }\int d(\xi,\eta )\textbf{1}_{\Omega_T} \tfrac {M_L(k,\xi)- M_L(l,\eta )}{M_L(k-l, \xi-\eta) M_L(l, \eta ) }\tfrac 1 {\vert k-l, \xi-\eta\vert^N}\tfrac {\xi l - k\eta} {((k-l)^2+(\xi-\eta-(k-l)\tau )^2 )^{\frac 12 }} \\
     &\qquad \qquad  (Aa^1)(k,\xi ) (Ap_i)(k-l, \xi-\eta )(A a^3) (l,\eta)\\
     &=T_{1,1}+ T_{1,2}+ T_{1,3} +T_2.
\end{align*}
Here, the  $T_{1,1}$ term  is due to estimate \eqref{eq:nuT1}. For \eqref{eq:nuT1j} we distinguish between the frequencies $\vert \tfrac \eta l -t \vert \ge\vert \tfrac {\xi-\eta}{k-l}-t \vert$ in $T_{1,2}$ and $\vert \tfrac \eta l -t \vert \le\vert \tfrac {\xi-\eta}{k-l}-t \vert$ in $T_{1,3}$. The $M_L$ commutator \eqref{eq:nuT2} is $T_2$, which requires further splitting. For  $T_{1,1}$ we use $\xi l - k\eta= (\xi-\eta-(k-l)t) l - (k-l)(\eta-lt )$, \eqref{eq:Mapp} and \eqref{eq:M2inv2} to estimate 
\begin{align*}
      T_{1,1}&= \sum_{k,l,k-l \neq 0 }\iint d(\xi,\eta)  \textbf{1}_{\Omega_T} \tfrac 1 {M_L(l,\eta )}  (\tfrac 1 {\vert l\vert }+\nu^{\frac 1 {12}})    \tfrac 1 {\vert k-l, \xi-\eta \vert ^{N-1 }}\tfrac {\xi l - k\eta} {((k-l)^2+(\xi-\eta-(k-l)\tau )^2 )^{\frac 12 }}\\
     &\qquad \qquad  (Aa^1)(k,\xi )  (Ap_i)(k-l, \xi-\eta )(A a^3) (l,\eta) \\
      &\lesssim   \Vert Aa^1_{\neq}\Vert_{L^2 } \Vert A p_{i,\neq} \Vert_{L^2} \Vert \tfrac 1 {M_L}Aa^3_{\neq} \Vert_{L^2 }+ \Vert Aa^1_{\neq}\Vert_{L^2 } \Vert A \Lambda_t^{-1} p_{i,\neq} \Vert_{L^2} \Vert  \tfrac 1 {M_L} A\p_y^ta^3_{\neq} \Vert_{L^2 }\\
      &\quad + \nu^{\frac 1 {12} }\Vert Aa^1_{\neq}\Vert_{L^2 } \Vert A p_{i,\neq} \Vert_{L^2} \Vert \tfrac 1 {M_L}A\p_x a^3_{\neq} \Vert_{L^2 } \\
      &\le  L \Vert Aa^1_{\neq}\Vert_{L^2 } \Vert A p_{i,\neq} \Vert_{L^2}  \Vert Aa^3 _{\neq}\Vert_{L^2 }+ L \Vert Aa^1_{\neq}\Vert_{L^2 } \Vert A \Lambda_t^{-1} p_{i,\neq} \Vert_{L^2} \Vert A\p_y^ta^3_{\neq} \Vert_{L^2 }\\
      &\quad + \nu^{\frac 1 {12} }\Vert Aa^1_{\neq}\Vert_{L^2 } \Vert A p_{i,\neq} \Vert_{L^2} \Vert \Lambda_t   a^3_{\neq} \Vert_{L^2}.
\end{align*}
After integrating in time we deduce that
\begin{align*}
    \int  T_{1,1 } d\tau \lesssim (L+ \nu^{-\frac 1 {12}} )  \kappa^{-\frac 12 } \eps^3.
\end{align*}
For $T_{1,2}$ we use  $\vert \tfrac \eta l -t \vert \ge\vert \tfrac {\xi-\eta}{k-l}-t \vert $ to infer that $\vert \xi l - k\eta\vert  =\vert (\xi-\eta-(k-l)t) l - (k-l) (\eta-lt) \vert \le 2 \vert (k-l) (\eta-lt )\vert$. Furthermore, with $\sum_{i=\nu ,\kappa, \nu^3 } \vert M_i(k,\xi)- M_i(l,\eta)\vert\approx 1 $ and \eqref{eq:M2inv} we conclude that
\begin{align*}
    T_{1,2}&\lesssim\sum_{k,l,k-l \neq 0 }\int d(\xi,\eta )\textbf{1}_{\Omega_T} \textbf{1}_{\vert \frac \eta l -t \vert \ge\vert \frac {\xi-\eta}{k-l}-t \vert}\tfrac{1} {M_L(l,\eta )} \tfrac 1 {\vert k-l, \xi-\eta \vert^{N}}\tfrac {\vert (k-l) (\eta-lt )\vert} {((k-l)^2+(\xi-\eta-(k-l)\tau )^2 )^{\frac 12 }} \\
    &\qquad \qquad (Aa^1)(k,\xi )(Ap_i)(k-l, \xi-\eta )(A a^3) (l,\eta)\\
    &\lesssim  \Vert Aa^1_{\neq}\Vert_{L^2}\Vert A \Lambda_t^{-1}p_{i,\neq} \Vert_{L^2}\Vert\tfrac 1 {M_L}  A\p_y^t a^3_{\neq} \Vert_{L^2 } \\
    &\lesssim  L \Vert Aa^1_{\neq}\Vert_{L^2}\Vert A \Lambda_t^{-1}p_{i,\neq} \Vert_{L^2}\Vert A\p_y^t a^3_{\neq} \Vert_{L^2 } .
\end{align*}
So after integrating in time, we obtain 
\begin{align*}
    \int  T_{1,2} d\tau &\lesssim L \kappa^{-\frac 1 2 } \eps^3.
\end{align*}
For $T_{1,3 }$, we use $\vert \tfrac \eta l -t \vert \le\vert \tfrac {\xi-\eta}{k-l}-t \vert  $ to infer  $\xi l - k\eta \le 2(\xi-\eta-(k-l)t) l $. Furthermore, with \eqref{eq:Mjdiff} we deduce 
\begin{align*}
        \sum_{j=\kappa,\nu ,\nu^3}\vert M_j(k,\xi)- M_j(l,\eta)\vert &\lesssim \nu^{\frac 1 3 } \tfrac { \vert \xi l-k\eta\vert }{\vert kl \vert}.
\end{align*}
Combining these two estimates by Hölder's inequality and \eqref{eq:M2inv} it follows, that 
\begin{align*}
    T_{1,3 }&\lesssim \sum_{k,l,k-l \neq 0 }\int d(\xi,\eta )\textbf{1}_{\Omega_T} \textbf{1}_{\vert \frac \eta l -t \vert \le\vert \frac {\xi-\eta}{k-l}-t \vert} \tfrac{1} { M_L(l,\eta )} \tfrac 1 {\vert k-l, \xi-\eta \vert^{N}}\tfrac {\nu^{\frac 1 3}(\xi l - k\eta)^2} { k l((k-l)^2+(\xi-\eta-(k-l)\tau )^2 )^{\frac 12 }}\\
    &\qquad \qquad (Aa^1)(k,\xi ) (Ap_i)(k-l, \xi-\eta )(A a^3) (l,\eta)\\
    &\lesssim \sum_{k,l,k-l \neq 0 }\int d(\xi,\eta )\textbf{1}_{\Omega_T}\textbf{1}_{\vert \frac \eta l -t \vert \le\vert \frac {\xi-\eta}{k-l}-t \vert} \tfrac{1} { M_L(l,\eta )} \tfrac 1 {\vert k-l, \xi-\eta \vert^{N}}\tfrac {\nu^{\frac 1 3}(\xi-\eta-(k-l)t)^2 l^2} { k l((k-l)^2+(\xi-\eta-(k-l)\tau )^2 )^{\frac 12 }}\\
    &\qquad \qquad (Aa^1)(k,\xi ) (Ap_i)(k-l, \xi-\eta )(A a^3)(l,\eta)\\
    &\lesssim \nu^{\frac 1 3} \sum_{k,l,k-l \neq 0 }\int d(\xi,\eta )\textbf{1}_{\Omega_T}\textbf{1}_{\vert \frac \eta l -t \vert \le\vert \frac {\xi-\eta}{k-l}-t \vert} \tfrac{1} { M_L(l,\eta )} \tfrac{\vert \xi-\eta-(k-l)t\vert } {\vert k-l, \xi-\eta \vert^{N-1}}\\
    &\qquad \qquad (Aa^1)(k,\xi ) (Ap_i)(k-l, \xi-\eta )(A a^3)(l,\eta)\\
    &\lesssim \nu^{\frac 13 }\Vert Aa^1_{\neq}\Vert_{L^2} \Vert A \Lambda_t p_{i,\neq}\Vert_{L^2} \Vert\tfrac 1 {M_L}  Aa^3_{\neq}\Vert_{L^2}\\
    &\lesssim L  \nu^{\frac 13 }\Vert Aa^1_{\neq}\Vert_{L^2} \Vert A \Lambda_t p_{i,\neq}\Vert_{L^2} \Vert  Aa^3_{\neq}\Vert_{L^2}.
\end{align*}
Thus integrating in time yields 
\begin{align*}
    \int  T_{1,3 } d\tau & \le L \nu^{\frac 1 6} \kappa^{-\frac 1 2}\eps^3.
\end{align*}
To estimate the $T_2$ term, we split the integral into the sets
\begin{align*}
    \Omega_1 &=\{\min(t-\tfrac \eta l, t-\tfrac {\xi-\eta }{k-l})\ge \nu^{-1 }\},\\
    \Omega_2 &=\{t-\tfrac \eta l \ge \nu^{-1}\ge t-\tfrac {\xi-\eta }{k-l}\},\\
    \Omega_3 &=\{t-\tfrac {\xi-\eta }{k-l} \ge \nu^{-1 }\ge t-\tfrac \eta l\},\\
    \Omega_4 &=\{t-\tfrac \xi k \ge \nu^{-1} \ge \max(t-\tfrac \eta l, t-\tfrac {\xi-\eta }{k-l})\}.
\end{align*}
For frequencies such that $\nu^{-1 } \ge \max( t-\tfrac \eta l,t-\tfrac \xi k )$, then $M_L(k,\xi)-M_L(l,\eta )=0$ and hence the commutator vanishes. Thus the sets $\Omega_j$ covers all regions of the support. The sets $\Omega_1, \Omega_2$ and $\Omega_3$ are chosen to distinguish between $\tfrac 1{M_L}=1 $ and $\tfrac 1 {M_L}  > 1$ for different frequencies and on set $\Omega_4 $ we use strong dissipation in the first component. We split the set $T_2$ into
\begin{align*}
    T_2&= \sum_{k,l,k-l \neq 0 }\int d(\xi,\eta ) \textbf{1}_{\Omega_T} \tfrac {M_L(k,\xi)- M_L(l,\eta )}{M_L(k-l, \xi-\eta) M_L(l, \eta ) }\tfrac {\xi l - k\eta} {((k-l)^2+(\xi-\eta-(k-l)\tau )^2 )^{\frac 12 }} \tfrac 1 {\vert k-l, \xi-\eta\vert^N}\\
    &\qquad \qquad (Aa^1)(k,\xi ) (Ap_i)(k-l, \xi-\eta ) (A a^3) (l,\eta)(\textbf{1}_{\Omega_1}+\textbf{1}_{\Omega_2}+\textbf{1}_{\Omega_3}+\textbf{1}_{\Omega_4})\\
    &= T_{2,1 }+T_{2,2}+T_{2,3}+T_{2,4 }.
\end{align*}
For  $T_{2,1} $ we use \eqref{eq:M2inv} to deduce 
\begin{align*}
T_{2,1}&\le \nu^2\sum_{k,l,k-l \neq 0 }\iint d(\xi,\eta) 1_{\Omega_1}  \langle t-\tfrac \eta l \wedge \kappa^{-\frac 1 3}\rangle  \langle t-\tfrac {\xi-\eta }{k-l }\wedge \kappa^{-\frac 1 3} \rangle \tfrac {1}{\vert k-l,\xi -\eta   \vert^{N-1}}\\
     &\qquad \qquad \tfrac {(\xi-\eta-(k-l)t) l - (k-l)(\eta- lt)} {((k-l)^2+(\xi-\eta-(k-l)\tau )^2 )^{\frac 12 }}(Aa^1)(k,\xi ) (Ap_i)(k-l, \xi-\eta )(A a^3) (l,\eta) \\
     &\lesssim \nu^2 \Vert Aa^1_{\neq} \Vert_{L^2} \Vert (\Lambda_t\wedge \kappa^{-\frac 1 3}) A p_{i,\neq} \Vert_{L^2} \Vert A\Lambda_t a^3_{\neq} \Vert_{L^2 }\\
     &\quad +L \nu\Vert Aa^1_{\neq} \Vert_{L^2} \Vert (\Lambda_t\wedge \kappa^{-\frac 1 3})\Lambda^{-1}_t \Lambda^{-1} A p_{i,\neq}\Vert_{L^2} \Vert A \p_y^t a^3_{\neq} \Vert_{L^2 }
\end{align*}
and so 
\begin{align*}
   \int T_{2,1}d\tau &\lesssim L\nu^{\frac 5 6}  \kappa^{-\frac 1 2 }\eps^3 .
\end{align*}
Now we consider $T_{2,2} $. By \eqref{eq:M2inv} we infer
\begin{align*}
T_{2,2}&\le \sum_{k,l,k-l \neq 0 }\iint d(\xi,\eta) \textbf{1}_{\Omega_2 } \nu \langle t-\tfrac \eta l \wedge \kappa^{-\frac 1 3}\rangle  \tfrac {1}{ \vert k-l,\xi -\eta   \vert^{N-1} }\tfrac {(\xi-\eta-(k-l)t) l - (k-l)(\eta- lt)} {((k-l)^2+(\xi-\eta-(k-l)\tau )^2 )^{\frac 12 }}\\
     &\qquad \qquad (Aa^1)(k,\xi )(Ap_i)(k-l, \xi-\eta )(A a^3) (l,\eta) \\
     &\lesssim \nu\Vert Aa^1_{\neq} \Vert_{L^2} \Vert A p_{i,\neq} \Vert_{L^2} \Vert A\Lambda_t a^3_{\neq} \Vert_{L^2 }\\
     &\quad +L\Vert Aa^1_{\neq} \Vert_{L^2} \Vert A \Lambda^{-1}_t p_{i,\neq}\Vert_{L^2} \Vert A \p_y^t a^3_{\neq} \Vert_{L^2 }.
\end{align*}
Integrating in time yields  
\begin{align*}
    \int T_{2,2 }d\tau &\lesssim    L\kappa^{-\frac 1 2 }\eps^3 .
\end{align*}
To estimate $T_{2,3}$, we need to distinguish between different choices of $a$. Using \eqref{eq:M2inv} we estimate 
\begin{align*}
T_{2,3}&=\sum_{k,l,k-l \neq 0 }\iint d(\xi,\eta) \textbf{1}_{\Omega_3 } \nu\langle t-\tfrac{\xi-\eta}{k-l}\wedge \kappa^{-\frac 1 3}\rangle \tfrac {1}{\vert k-l,\xi -\eta   \vert^{N-1} }\tfrac {(\xi-\eta-(k-l)t) l - (k-l)(\eta- lt)} {((k-l)^2+(\xi-\eta-(k-l)\tau )^2 )^{\frac 12 }}\\
     &\qquad \qquad (Aa^1)(k,\xi )(Ap_i)(k-l, \xi-\eta )(A a^3) (l,\eta) \\
     &\lesssim \nu \Vert Aa^1_{\neq} \Vert_{L^2} \Vert (\Lambda_t \wedge \kappa^{-\frac 1 3 } ) A p_{i,\neq}  \Vert_{L^2} \Vert A\p_x a^3_{\neq} \Vert_{L^2 }\\
     &\quad +  \nu \Vert Aa^1_{\neq} \Vert_{L^2} \Vert A  (\Lambda_t \wedge \kappa^{-\frac 1 3 } ) \Lambda^{-1} \Lambda^{-1}_t p_{i,\neq}\Vert_{L^2} \Vert A \p_y^t a^3_{\neq} \Vert_{L^2 }
\end{align*}
and thus after integrating in time
\begin{align*}
    \int T_{2,3}[vvv]d\tau &\lesssim  \nu^{-\frac 12  }\eps^3,\\
    \int T_{2,3}[bvb]d\tau &\lesssim   \kappa^{-\frac 12 }\eps^3,\\
    \int T_{2,3}[bbv]d\tau &\lesssim  \kappa^{-\frac 12 }\eps^3.
\end{align*}
In the case of $vbb$, we use \eqref{eq:M2inv} to estimate
\begin{align*}
T_{2,3}[vbb]&=\sum_{k,l,k-l \neq 0 }\iint d(\xi,\eta)  \nu \langle t-\tfrac{\xi-\eta}{k-l}\rangle \tfrac {1}{\vert k-l,\xi -\eta   \vert^{N-1} }\tfrac {(\xi-\eta-(k-l)t) k - (k-l)(\xi- kt)} {((k-l)^2+(\xi-\eta-(k-l)\tau )^2 )^{\frac 12 }}\\
     &\qquad \qquad (Av)(k,\xi ) (Ap_2)(k-l, \xi-\eta )(A b) (l,\eta) \\
     &\lesssim \nu  \Vert\p_x  Av_{\neq} \Vert_{L^2} \Vert(\Lambda_t\wedge\kappa^{-\frac 1 3 }) A p_{2,\neq }  \Vert_{L^2} \Vert Ab_{\neq} \Vert_{L^2 }\\
     &\quad + \nu  \Vert \p_y^t Av_{\neq} \Vert_{L^2} \Vert (\Lambda_t\wedge\kappa^{-\frac 1 3 })\Lambda^{-1}_tA p_{2,\neq}\Vert_{L^2} \Vert A  b_{\neq} \Vert_{L^2 }\\
     &\le \nu \Vert Ab_{\neq} \Vert_{L^2} \Vert A \nabla_t v_{\neq} \Vert_{L^2 }\Vert  A\Lambda_t  p_{2,\neq }  \Vert_{L^2}.
\end{align*}
Thus after integrating in time, we obtain  
\begin{align*}
    \int T_{2,3} [vbb] d\tau&\lesssim \nu^{\frac 1 2 } \kappa^{-\frac 12 }\eps^3.
\end{align*}
For $T_{2,4}$ we obtain that $M(l,\eta) =M(k-l , \xi-\eta)=1 $. We use  $t-\tfrac \xi k \ge \nu\ge \max(t-\tfrac \eta l,t-\tfrac {\xi-\eta} {k-l } )$ to deduce that
\begin{align*}
    1= \tfrac {kt-\xi}{kt-\xi}= \tfrac {k}k \tfrac {t-\frac \xi k }{t-\frac \xi k } \le \nu\tfrac {\vert \xi-  kt \vert} {\vert k \vert }.
\end{align*}
With $\xi l - k\eta=(\xi - \eta -(k-l)t)l - (k-l)(\eta-lt)$ we infer that 
\begin{align*}
T_{2,4}&=\sum_{k,l,k-l \neq 0 }\iint d(\xi,\eta)   \tfrac 1{ \vert k-l,\xi -\eta   \vert^{N} } \tfrac {\xi l - k\eta } {((k-l)^2+(\xi-\eta-(k-l)\tau )^2 )^{\frac 12 }}\\
     &\qquad \qquad (Aa^1)(k,\xi )(Ap_i)(k-l, \xi-\eta )(A a^3) (l,\eta) \\
     &=\nu\sum_{k,l,k-l \neq 0 }\iint d(\xi,\eta)  \vert \xi-  kt \vert  \tfrac {1}{ \vert k-l,\xi -\eta   \vert^{N-1} } \tfrac {(\xi-\eta-(k-l)t) l } { \vert l\vert ((k-l)^2+(\xi-\eta-(k-l)\tau )^2 )^{\frac 12 }}\\
     &\qquad \qquad  (Aa^1)(k,\xi ) (Ap_i)(k-l, \xi-\eta )(A a^3) (l,\eta) \\
     &\quad +\sum_{k,l,k-l \neq 0 }\iint d(\xi,\eta)  \tfrac {1}{ \vert k-l,\xi -\eta   \vert^{N-1} }\tfrac {(k-l ) (\eta-lt ) } {  ((k-l)^2+(\xi-\eta-(k-l)\tau )^2 )^{\frac 12 }}\\
     &\qquad \qquad (Aa^1)(k,\xi )(Ap_i)(k-l, \xi-\eta )(A a^3) (l,\eta) \\
     &\le \nu \Vert A\p_y^t  a^1_{\neq}\Vert _{L^2} \Vert A p_{i,\neq}\Vert_{L^2} \Vert Aa^3_{\neq} \Vert_{L^2} \\
     &\quad +\Vert Aa^1_{\neq}\Vert_{L^2} \Vert \Lambda_t^{-1} Ap_{i,\neq } \Vert_{L^2} \Vert A\p_y^t a^3_{\neq} \Vert_{L^2}.
\end{align*}
Thus integrating in time yields 
\begin{align*}
    \int  T_{2,4} d\tau&\lesssim \kappa^{-\frac 1 2 }\eps^3.
\end{align*}

\subsection{Nonlinear terms with an $x$-average in the second component}\label{sec:second}
We apply the notation of \eqref{eq:a}
\begin{align*}
    &\quad \langle A a^1_{\neq} , A(a^2_{1,=}\p_xa^3_{\neq}) -a^2_{1,=}\p_x  Aa^3_{\neq} \rangle\\
    &= \sum_{k\neq 0} \iint d(\xi,\eta ) (A a^1)(k,\xi) (A(k,\xi)-A(k,\eta )) k a^2_{1}(0,\xi-\eta)a^3(k,\eta) \\
    &= R+T. 
\end{align*}
Here we split into reaction and transport terms according to the sets
\begin{align*}
    \Omega_R &= \{ \vert \xi-\eta \vert \ge \tfrac 1 8 \vert k, \eta \vert\},\\
    \Omega_T &= \{ \vert \xi-\eta \vert <\tfrac 1 8\vert k, \eta \vert   \}.
\end{align*}
\textbf{Reaction term} On the set $\Omega_R $ it holds that $\vert \xi-\eta \vert\ge  \tfrac 1 8 \vert k ,\eta\vert $, then we obtain $\vert A(k,\xi)-A(k,\eta )\vert \lesssim \vert \xi-\eta \vert^N $ and thus  with \eqref{eq:M2inv}, it follows that
\begin{align*}
     \sum_{k\neq 0} \iint d(\xi,\eta ) &(A a^1)(k,\xi) (A(k,\xi)-A(k,\eta )) k a^2_{1}(0, \xi-\eta)a^3(k,\eta) \\ 
    &\lesssim  \Vert A a^1_{\neq}\Vert_{L^2} \Vert  a^2_{=}\Vert_{H^N} \Vert\p_x a^3_{\neq}\Vert_{L^\infty }\\
    &\lesssim  \Vert A a^1_{\neq}\Vert_{L^2} \Vert  a^2_{=}\Vert_{H^N} \Vert\tfrac 1 {M_L}  A a^3_{\neq}\Vert_{L^2}\\
    &\lesssim L \Vert A a^1_{\neq}\Vert_{L^2} \Vert  a^2_{=}\Vert_{H^N} \Vert A a^3_{\neq}\Vert_{L^2}.
\end{align*}
Integrating in time yields a bound
\begin{align*}
    \int R \ d\tau &\lesssim L \kappa^{-\frac 1 3 } \eps^2 \tilde \eps  .
\end{align*}
\textbf{Transport term} On the set $\Omega_L$ it holds that $\vert k ,\eta\vert\ge  \tfrac 1 8 \vert \xi-\eta \vert $. By the mean value theorem there exists a $\theta \in[0,1] $ 
\begin{align*}
    \left \vert  \vert k,\eta \vert^N-  \vert k,\xi \vert^N\right\vert  \lesssim \vert \xi-\eta \vert \vert k,\eta-\theta \xi  \vert^{N-1} \lesssim \vert \xi-\eta \vert \vert k,\eta \vert^{N-1}. 
\end{align*}
Thus, we can estimate the difference in $A$ by 
\begin{align*}
    \vert A(k,\xi )-A(k,\eta ) \vert &\lesssim (M_L(k,\xi ) - M_L(k,\eta ) )\vert k,\xi \vert^N \\
    &\quad +M_L(k,\eta ) \vert k,\xi  \vert^N \sum_{j=1,\kappa,\nu ,\nu^3} \vert M_j (k,\xi )- M_j(k,\eta )\vert\\
    &\quad + M_L(k,\eta ) (  \vert k,\eta \vert^N-  \vert k,\xi \vert^N) \\
    &\lesssim  \tfrac 1 k   \vert \xi-\eta \vert  \vert  k,\xi \vert^N
\end{align*}
where we used \eqref{eq:M1diff1},\eqref{eq:M2diff} and \eqref{eq:Mjdiff}  to estimate the differences in $M_j$. So we infer, that
\begin{align*}
    T&\le \sum_{k\neq 0 } \iint d(\xi,\eta ) \textbf{1}_{\Omega_T }\vert A a^1\vert (k,\xi) \vert  a^2_{1}\vert (0,\xi-\eta)\tfrac 1{M_L(k,\eta)}\vert Aa^3\vert (k,\eta).   
\end{align*}
and thus integrating in time yields 
\begin{align*}
    \int T d\tau &\lesssim L \Vert A a^1_{\neq}\Vert_{L^2L^2}\Vert a^2_=\Vert_{L^\infty H^N}\Vert A a^3_{\neq}\Vert_{L^2L^2}\lesssim L \kappa^{-\frac 13 }\eps^2\tilde \eps.
\end{align*}

\subsection{Nonlinear terms with an $x$-average in the third component}\label{sec:third}
We aim to estimate
\begin{align*}
    &\quad \langle A a^1_{1,\neq} , A(a^2_{2,\neq}\p_y a^3_{1,=}) \rangle\\
    &= \sum_{k\neq 0} \iint d(\xi,\eta ) (A a^1_1)(k,\xi) A(k,\xi)\tfrac {k\eta }{\sqrt{k^2+(\xi-\eta -kt)^2}}p_i(k,\xi-\eta) a^3_1(0,\eta ) \\
    &= R+ T
\end{align*}
where we split into the reaction and transport terms according to the sets
\begin{align*}
    \Omega_R &= \{ \vert k,\xi-\eta \vert \ge \tfrac 1 8 \vert \eta \vert\},\\
    \Omega_T &= \{ \vert k,\xi-\eta \vert <\tfrac 1 8\vert \eta \vert   \}.
\end{align*}
\textbf{Reaction term} On the set $\Omega_R $ it holds that $\vert k,\xi-\eta \vert\ge \tfrac 1 8 \vert \eta\vert$. With \eqref{eq:M2inv} we infer
\begin{align*}
    R&=\sum_{k\neq 0} \iint d(\xi,\eta ) \textbf{1}_{\Omega_R}(A a^1_1)(k,\xi) A(k,\xi)\tfrac {k\eta }{\sqrt{k^2+(\xi-\eta -kt)^2}}p_i(k,\xi-\eta) a^3_1(0,\eta ) \\
    &\lesssim  \Vert A a^1_{1,\neq}\Vert_{L^2} \Vert A \tfrac 1 {M_L}  \p_x \Lambda^{-1}_t p_{i,\neq} \Vert_{L^2}\Vert \p_y a^3_{1,=}\Vert_{L^\infty }\\
    &\lesssim  \Vert A a^1_{1,\neq}\Vert_{L^2} \Vert A  p_{i,\neq} \Vert_{L^2} \Vert   a^3_{1,=}\Vert_{H^N  }.
\end{align*}
Integrating in time then yields 
\begin{align*}
    \int Rd\tau &\lesssim \kappa^{-\frac 1 3 } \eps^2\tilde \eps.
\end{align*}
\textbf{Transport term } On the set $\Omega_T$ it holds that $\vert k,\xi-\eta \vert\le \tfrac 1 8 \vert \eta\vert$, then with \eqref{eq:M2inv} we estimate
\begin{align*}
    T&= \sum_{k\neq 0} \iint d(\xi,\eta ) \textbf{1}_{\Omega_T} (A a^1_1)(k,\xi) A(k,\xi)\tfrac {k\eta }{\sqrt{k^2+(\xi-\eta -kt)^2}}p_i(k,\xi-\eta) a^3_1(0,\eta ) \\
    &\lesssim  \Vert A a^1_{\neq}\Vert_{L^2}\Vert \p_x \Lambda^{-1}_t p_{i,\neq} \Vert_{L^\infty } \Vert \p_y a^3_=\Vert_{H^N}\\
    &\lesssim  \Vert A a^1_{\neq}\Vert_{L^2}\Vert  \tfrac 1 {M_L} A \Lambda^{-1}_t p_{i,\neq} \Vert_{L^2 } \Vert \p_y a^3_=\Vert_{H^N}\\
    &\lesssim  \Vert A a^1_{\neq}\Vert_{L^2}(\Vert   \Lambda^{-1}_t  Ap_i \Vert_{L^2}+ \nu \Vert   Ap_{i,\neq} \Vert_{L^2}) \Vert \p_y a^3_=\Vert_{H^N}.
\end{align*}
Integrating in time yields 
\begin{align*}
    \int T  d\tau &\lesssim  \kappa^{-\frac 1 2 } \eps^2\tilde \eps .
\end{align*}

\subsection{Nonlinear terms with an $x$-average in first component}\label{sec:first}
Now we turn to 
\begin{align*}
    &\quad \langle\langle \p_y\rangle^N  a^1_{=,1} , \langle \p_y\rangle^N  (a^2_{\neq}\nabla_t a^3_{\neq,1 })_=\rangle\\ 
    &=-\sum_{k\neq 0} \iint d(\xi,\eta ) \langle \xi \rangle^{2N}  a^1_{1}(0,\xi)\tfrac {k \xi }{\sqrt{k^2+ (\xi-\eta+kt)^2}}p_i(-k ,\xi-\eta)a^3_1 (k,\eta ).
\end{align*}
Applying Hölder's inequality, the Sobolev embedding and the definition of $A$ yields 
\begin{align*}
&\quad \langle\langle \p_y\rangle^N  a^1_{1,=} , \langle \p_y\rangle^N  (a^2_{\neq}\nabla_t a^3_{\neq,1 })_=\rangle\\ 
    &\le \Vert \p_y \langle \p_y \rangle^{N}  a^1_= \Vert_{L^2} ( \Vert \p_x \Lambda^{-1}_t p_{i,\neq} \Vert_{L^\infty }\Vert \langle \p_y \rangle^N a^3_{\neq} \Vert_{L^2 } +\Vert \langle \p_y \rangle^N \Lambda^{-1}_t p_{i,\neq} \Vert_{L^2}\Vert \p_x a^3_{\neq} \Vert_{L^\infty }\\
    &\le \Vert \p_y   a^1_= \Vert_{H^N} \Vert A \tfrac 1 {M_L }  \Lambda^{-1}_t p_{i,\neq} \Vert_{L^2 }\Vert A \tfrac 1 {M_L } a^3_{\neq} \Vert_{L^2}  
\end{align*}
With \eqref{eq:M2inv} we infer 
\begin{align*}
&\quad \langle\langle \p_y\rangle^N  a^1_{1,=} , \langle \p_y\rangle^N  (a^2_{\neq}\nabla_t a^3_{\neq,1 })_=\rangle\\ 
    &\lesssim \Vert \p_y   a^1_= \Vert_{H^N}( \Vert A   \Lambda^{-1}_t p_{i,\neq} \Vert_{L^2 }+ \nu \Vert A p_{i,\neq} \Vert_{L^2} ) (\Vert A a^3_{\neq} \Vert_{L^2 }+ \nu \Vert A (\Lambda_t \wedge \kappa^{-\frac 1 3 }) a^3_{\neq} \Vert_{L^2 } )  \\
    &\lesssim \Vert \p_y  a^1_= \Vert_{H^N}  \Vert A  \Lambda^{-1}_t p_{i,\neq} \Vert_{L^2 } (\Vert A  a^3_{\neq} \Vert_{L^2} +\nu \Vert (\Lambda_t \wedge \kappa^{-\frac 1 3 }) a^3_{\neq}\Vert_{L^2} )\\
    &\quad +  \nu  \Vert \p_y  a^1_= \Vert_{H^N} \Vert A  p_{i,\neq}\Vert_{L^2} (\Vert A  a^3_{\neq} \Vert_{L^2} +\nu \Vert (\Lambda_t \wedge \kappa^{-\frac 1 3 }) a^3_{\neq}\Vert_{L^2} ).
\end{align*}
Integrating in time yields 
\begin{align*}
    \int \langle\langle \p_y\rangle^N  a^1_{=} , \langle \p_y\rangle^N  (a^2_{\neq}\nabla_t a^3_{\neq,1 })_=\rangle d\tau 
    \lesssim L \kappa^{-\frac 1 2 }\eps^2\tilde \eps .
\end{align*}

\subsection{Other nonlinear terms}\label{sec:ONL}

In this subsection, we aim to estimate
\begin{align*}
    \langle \chi A \p_y^t \Delta_t^{-1}  a^1_{\neq} , A(a^2\nabla_t a^3) \rangle&=\langle \chi A \p_y^t \Delta_t^{-1}  a^1_{\neq} , A(a^2_{\neq} \nabla_t a^3_{\neq} ) \rangle\\
    &+\langle \chi A \p_y^t \Delta_t^{-1}  a^1_{\neq} , A(a^2_{\neq} \nabla_t a^3_=) \rangle\\
    &+\langle \chi A \p_y^t \Delta_t^{-1}  a^1_{\neq} , A(a^2_=\nabla_t a^3_{\neq} ) \rangle.
\end{align*}
with the choices $a^1a^2a^3\in \{bvv,bbb,vbv,vvb\}$. We start with the case of no $x$-averages and use 
\begin{align*}
    \xi l - k \eta &= (\xi-kt )(l-k)+ k(\xi- \eta -(k-l)t ) 
\end{align*}
and \eqref{eq:M2inv} to infer 
\begin{align*}
&\quad  \langle \chi A \p_y^t \Delta_t^{-1}  a^1_{\neq} , A(a^2_{\neq}\nabla_t a^3_{\neq}) \rangle\\
&= \sum_{k,l,k-l \neq 0 }\iint d(\xi,\eta) \tfrac {\xi-kt}{k^2+(\xi-kt)^2 }   \tfrac {\xi l -k\eta} {\sqrt{(k-l )^2+ (\xi-\eta -(k-l )t)^2 }}A^2 (k,\xi) a^1(k,\xi )p_i(k-l, \xi-\eta) a^3(l,\eta) \\
    &\lesssim \Vert Aa^1_{\neq} \Vert_{L^2} \Vert \tfrac 1 {M_L }\p_x \Lambda^{-1}_t A  p_{i,\neq}  \Vert_{L^2} \Vert \tfrac 1 {M_L } A a^3_{\neq} \Vert_{L^2} \\
    &\quad +\Vert \p_x \Lambda^{-1}_t  Aa^1_{\neq} \Vert_{L^2}\Vert  \tfrac 1 {M_L }A  p_{i,\neq}  \Vert_{L^2} \Vert \tfrac 1 {M_L } A a^3_{\neq} \Vert_{L^2} \\
    &\lesssim L \Vert Aa^1_{\neq} \Vert_{L^2} (\Vert \p_x \Lambda_t^{-1} A  p_{i,\neq}  \Vert_{L^2}+ \nu \Vert A  p_{i,\neq}  \Vert_{L^2}) \Vert A a^3_{\neq} \Vert_{L^2}\\
    &\quad +L(1+ \nu \kappa^{-\frac 1 3 }) \Vert \p_x \Lambda^{-1}_t  Aa^1_{\neq} \Vert_{L^2} \Vert A p_{i,\neq} \Vert_{L^2}  \Vert A a^3_{\neq} \Vert_{L^2}.
\end{align*}
Thus integrating in time yields 
\begin{align*}
    \quad\int \langle \chi A \p_y^t \Delta_t^{-1}  a^1_{\neq} , A(a^2_{\neq}\nabla_t a^3_{\neq}) \rangle d\tau\lesssim L  \kappa^{-\frac 12} \eps^3. 
\end{align*}
For the case, when the average is in the second component, we use partial integration to estimate 
\begin{align*}
    &\quad \langle \chi A \p_y^t \Delta_t^{-1}  a^1_{\neq} , A(a^2_{1, =} \p_x a^3_{\neq}) \rangle\\
    &=-\langle \chi A \p_x \p_y^t \Delta_t^{-1}  a^1_{\neq} , A(a^2_{1, =}  a^3_{\neq}) \rangle\\
    &\lesssim \Vert  \p_x \Lambda^{-1}_t a^1_{\neq} \Vert_{L^2}\Vert a^2_=\Vert_{H^N}\Vert\tfrac 1{M_L }A a^3_{\neq}\Vert_{L^2}\\
    &\lesssim L  \Vert  \p_x \Lambda^{-1}_t a^1_{\neq} \Vert_{L^2}\Vert  a^2_=\Vert_{H^N }\Vert   A a^2_{\neq}\Vert_{L^2}
\end{align*}
and thus integrating in time and using $L= \max(1, \nu\kappa^{-\frac 1 3 }) $ yields
\begin{align*}
    \int\langle \chi A \p_y^t \Delta_t^{-1}  a^1_{\neq} , A(a^2_{1, =} \p_x a^3_{\neq}) \rangle d\tau&\lesssim \kappa^{-\frac 1 2 }\eps^2\tilde \eps.
\end{align*}
For the case when the average is in the third component, we obtain
\begin{align*}
    &\quad \langle\chi A \p_y^t \Delta_t^{-1}  a^1_{\neq} , A(a^2_{2,\neq}\p_y a^3_=) \rangle\\ 
    &= \sum_{k\neq  0} \iint d(\xi,\eta )\chi \tfrac {\xi-kt }{k^2+(\xi-kt)^2} \tfrac {k\eta } {\sqrt{k^2+(\xi-\eta -t)^2 }}\tfrac {A(\xi,k )}{A(\xi-\eta,k) }(Aa^1)(k,\xi ) (Ap_i)(k,\xi-\eta ) a^3(0,\eta ).
\end{align*}
Thus by $\eta = \xi -kt - (\xi-\eta -kt) $ we estimate
\begin{align*}
    &\quad \langle\chi A \p_y^t \Delta_t^{-1}  a^1_{\neq} , A(a^2_{2,\neq}\p_y a^3_=) \rangle\\ 
    &\le \Vert (\p_y^t)^2\Delta_t^{-1} A a^1_{\neq} \Vert_{L^2} \Vert \p_x \Lambda^{-1}_t \tfrac 1{M_L } A  p_{i,\neq} \Vert_{L^2} \Vert a^3_=\Vert_{H^N} \\ 
    &+\Vert \p_x \p_y^t\Delta_t^{-1} A a^1_{\neq} \Vert_{L^2} \Vert   \tfrac 1{M_L }A p_{i,\neq} \Vert_{L^2} \Vert a^3_=\Vert_{H^N} \\ 
    &\le L\Vert  A a^1_{\neq} \Vert_{L^2} \Vert \p_x \Lambda^{-1}_t A  p_{i,\neq} \Vert_{L^2} \Vert a^3_=\Vert_{H^N} \\ 
    &+L\Vert \p_x \Lambda_t^{-1} A a^1_{\neq} \Vert_{L^2} \Vert   A p_{i,\neq} \Vert_{L^2} \Vert a^3_=\Vert_{H^N} 
\end{align*}
Integrating in time and using $L= \max(1, \nu\kappa^{-\frac 1 3 }) $ yields
\begin{align*}
    \int\langle \chi A \p_y^t \Delta_t^{-1}  a^1_{\neq} , A(a^2_{2,\neq}\p_y a^3_=) \rangle  d\tau &\lesssim \kappa^{-\frac 12 }\eps^2 \tilde \eps.
\end{align*}
Which concludes the estimate 
\begin{align*}
    \int   ONL  d\tau &\lesssim  L \kappa^{-\frac 1 2 } \tilde \eps\eps^2.
\end{align*}
Combining the estimates of Subsection \ref{sec:no} to \ref{sec:ONL} completes the proof of Proposition \ref{prop:errors}and thus Theorem \ref{thm:comp}. \qed\\

In this article, we have shown that the MHD equations around Couette flow with magnetic resistivity smaller than fluid viscosity $\nu\ge \kappa >0$ are stable for initial data which is small enough in Sobolev spaces. If the resistivity is much smaller than the viscosity, $\nu \kappa^{-\frac 1 3 }>0$, large viscosity destabilizes the equation, leading to norm inflation of size $\nu \kappa^{-\frac 1 3 }$. Controlling this norm inflation is a major new challenge compared to other dissipation regimes.

\appendix

\section{Construction of the Weights } \label{app:weight}
Let $A$ be the Fourier weight 
\begin{align*}
    A:&= M \langle \nabla \rangle e^{c\kappa^{\frac 1  3} t \textbf{1}_{\neq} },
\end{align*}
with  $M=M_1M_L M_\kappa  M_{\nu} M_{\nu^3} $ defined as 
\begin{align*}    
    \tfrac {-\dot M_L } {M_L } &=    \tfrac {t-\frac \xi k  }{1+(\frac \xi k-t)^2} \textbf{1}_{\{\nu ^{-1 } \le t-\frac \xi k \le (c_1\kappa k^2) ^{-\frac 1 3 } \} }&&k\neq 0,\\
    \tfrac {-\dot M_1 } {M_1} &= C_\alpha  \tfrac { \vert k \vert+ \nu^{\frac1{12}}\vert k\vert^2  }{k^2 +(\xi-kt)^2}&&k\neq 0,\\
    \tfrac {-\dot M_\nu  } {M_\nu } &=\tfrac {\nu^{\frac 1 3 } }{1+ \nu^{\frac 2 3 } (t-\frac \xi k )^2} &&k\neq 0,\\ 
    \tfrac {-\dot M_\kappa } {M_{\kappa} } &= \tfrac {\kappa^{\frac 1 3 } }{1+ \kappa^{\frac 2 3 } (t-\frac \xi k )^2}&&k\neq 0,\\
    \tfrac {-\dot M_{\nu^3} } {M_{\nu^3}} &= \tfrac {C_\alpha  \nu}{1+ \nu^2(t-\frac \xi k )^2}&&k\neq 0,\\
    M_\cdot (t=0)&= M_\cdot (k=0)=1. 
\end{align*}
The weight $M_L$ is an adaption of the weight $m^{\frac 1 2}$ in \cite{liss2020sobolev} to the present setting and $M_{\nu^3}$ we use to differentiate between resonant and non-resonant regions. The method of using time-dependent Fourier weights is common when working at solutions around Couette flow and the other weights are modifications of previously used weights (cf. \cite{bedrossian2016sobolev,masmoudi2022stability,liss2020sobolev,zhao2023asymptotic} for shear related systems such as Navier-Stokes). The constants $C_\alpha = \tfrac 2{\min(1, \alpha -\frac 1 2)  }, c= \tfrac 1{20} (1-\tfrac 1{2\alpha })^2$ and $ c_1= \tfrac 1{20} (1-\tfrac 1{2\alpha })$ are determined through the linear estimates. For the weights we obtain that for all times $t>0$, it holds that 
\begin{align*}
    M_1\approx M_\kappa\approx M_{\nu} \approx M_{\nu^3}  &\approx 1 ,\\
    L^{-1}\le \min(1, \nu^{-1 }\kappa^{\frac 1 3 }k^{\frac 2 3 } ) \lesssim M_L &\le 1.
\end{align*}

\begin{lemma}[$M_L$ properties ]
The weight $M_L$ satisfies the following bounds 
    \begin{align}
    \textbf{1}_{\vert t-\frac \xi k \vert\ge\nu^{-1  }} \tfrac {t-\frac \xi k }{1+(t-\frac \xi k )^2}&\le \tfrac {-\dot M_L}{M_L }+\kappa k^2 c_1(1+(t-\tfrac \xi k )^2),\label{eq:M2est}\\
        \tfrac 1 {M_L(k,\xi) }&\le 1+ \nu^{\frac 1 2 } \langle t- \tfrac \xi k \rangle \wedge \kappa^{-\frac 13 }. \label{eq:M2inv}
    \end{align}
Furthermore, it follows for $a\in H^1 $, that
\begin{align}
    \Vert \tfrac 1 {M_L } a_{\neq}\Vert_{L^2}&\le \Vert  a_{\neq}\Vert_{L^2}+ \Vert (\Lambda^{-1}_t \wedge \kappa^{-\frac 1 3 }) a_{\neq}\Vert_{L^2},\label{eq:M2inv1}\\
    \Vert \tfrac 1 {M_L } \p_x  a_{\neq}\Vert_{L^2}&\le \Vert  \Lambda_t a_{\neq}\Vert_{L^2}.\label{eq:M2inv2}
\end{align}

\end{lemma}
\begin{proof}
    This follows immediately, from the definition of $M_L$.  
\end{proof}

\begin{lemma}[Enhanced dissipation estimates ]
The weights $M_\nu$ and $M_\kappa$ satisfy the following bounds 
    \begin{align}
        \tfrac 1 2 \nu^{\frac 1 3}&\le \tfrac {-\dot M_\kappa } {M_\kappa} + \nu (k^2 + (\xi -kt)^2),\label{eq:M3est} \\
        \tfrac 1 2 \kappa^{\frac 1 3}&\le \tfrac {-\dot M_{\nu} } {M_{\nu}} + \kappa (k^2 + (\xi -kt)^2)\label{eq:M4est}.
    \end{align}

\end{lemma}
\begin{proof}
    This follows immediately, from the definition of $M_\nu$ and $M_\kappa$.  
\end{proof}

\begin{lemma}[Difference estimates] \label{lem:Mdiff}Let $k,l \in\bbZ\setminus \{0\}$ and $\xi,\eta \in \bbR $, then there hold the following bounds on differences
    \begin{align}
        1-\tfrac {M_1(k,\xi)} {M_1(k,\eta)}&\lesssim \tfrac {\vert \xi-\eta\vert}{\vert k\vert },\label{eq:M1diff1}\\
        1-\tfrac {M_1(k,\xi)} {M_1(l,\eta)}&\lesssim \tfrac { \vert k-l \vert } {\vert l\vert}+ \nu^{\frac 1 {12} },\label{eq:M1diff}\\
       M_L(k,\eta)- M_L(k,\xi) &\le 2 \tfrac {\vert \xi-\eta\vert} k, \label{eq:M2diff}\\
        1-\tfrac {M_j(k,\xi)} {M_j(l,\eta)} &\le 2 j^{\frac 1 3 } \tfrac { \vert \xi l-k\eta\vert }{\vert kl \vert},\qquad \qquad j \in \{\kappa ,\nu, \nu^3\}\label{eq:Mjdiff}
    \end{align}
\end{lemma}

\begin{proof}
We start with the $M_1$ estimate \eqref{eq:M1diff1} and consider $M_1(k,\xi )\le M_1(k ,\eta ) $
\begin{align*}
    1-\tfrac {M_1(k,\xi)} {M_1(k,\eta)}&= 1-\exp \left (-\left \vert \int_0^t\tfrac {\vert k \vert + \vert k \vert ^2\nu^{\frac 1 {12}}}{k^2 +(\xi-k\tau )^2}- \tfrac {\vert k \vert+ \vert k \vert ^2\nu^{\frac 1 {12}}}{k^2 +(\eta -kt)^2}d\tau \right \vert \right ),\\
    &\le 1- \exp\left(- \int_{[\frac \xi k, \frac \eta k]\cup t- [\frac \xi k, \frac \eta k]} 1  \  d\tau \right)\lesssim \tfrac {\vert \xi-\eta\vert}{\vert k\vert }.
\end{align*}
The case $M_1(k,\xi )\ge M_1(k,\eta ) $ follows by the same argument and $M_1(k,\xi )\approx  M_1(k,\eta ) \approx 1 $. For \eqref{eq:M1diff} we consider the case  $M_1(k,\xi )\le M_1(l,\eta ) $ and infer that 
\begin{align*}
    1-\tfrac {M_1(k,\xi)} {M_1(l,\eta)}&= 1-\exp \left (-\left \vert \int_0^t\tfrac {\vert k \vert + \vert k \vert ^2\nu^{\frac 1 {12}}}{k^2 +(\xi-kt)^2}- \tfrac {\vert l \vert+ \vert l \vert ^2\nu^{\frac 1 {12}}}{l^2 +(\eta -lt)^2}\right \vert \right ),\\
    &= 1-\exp(-2\pi(\tfrac 1 {l\wedge k} + \nu^{\frac 1 {12} }))\lesssim \tfrac 1 {l\wedge k}+\nu^{\frac 1 {12} }\lesssim \tfrac { \vert k-l \vert } {\vert l\vert}+ \nu^{\frac 1 {12} }. 
\end{align*}
The case $M_1(k,\xi )\ge M_1(l,\eta ) $ follows by the same argument and $M_1(k,\xi )\approx  M_1(l,\eta ) \approx 1 $. For \eqref{eq:M2diff} we consider the case $M_L(k,\xi)\le M_L(k,\eta)$  and thus 
\begin{align*}
     M_L(k,\eta)- M_L(k,\xi)&= M_L(k,\eta)(1-\tfrac {M_L(k,\xi)} {M_L(k,\eta)})\lesssim 1-\tfrac {M_L(k,\xi)} {M_L(k,\eta)}.
\end{align*}
We infer 
\begin{align*}
    1-\tfrac {M_L(k,\xi)} {M_L(k,\eta)}&= 1-\exp\left(-\left\vert \int_0^t\tfrac {\tau-\frac \xi k  }{1 +(\tau-\frac \xi k )^2}d\tau- \int_0^t\tfrac {\tau-\frac \eta k  }{1 +(\tau-\frac \eta k )^2}d\tau\right \vert \right ), \\
    &= 1-\exp\left(- \int_{-[\frac \xi k , \frac \eta k ]\cup t-[\frac \xi k , \frac \eta k ]} 1 d\tau \right ), \\
    &\le  2 \tfrac {\vert \xi -\eta \vert } k. 
\end{align*}
The case $M_L(k,\xi)\ge M_L(k,\eta)$ follows by the same argument. 

For \eqref{eq:Mjdiff} we estimate the $M_\kappa$ difference, since the $M_{\nu}$ and $M_{\nu^3}$ differences are done similar.  Let   $M_\kappa(k,\xi)\ge M_\kappa(l,\eta )$, then it follows
\begin{align*}
    1- \tfrac {M_\kappa(k,\xi)}{M_\kappa(l,\eta )}&= 1-\exp\left (-\kappa^{\frac 1 3 } \vert \int_0^t \tfrac 1 {1+\kappa^{\frac 2 3 }(t-\frac \xi k )^2}-\tfrac 1 {1+\kappa^{\frac 2 3 }(t-\frac \eta l  )^2}\vert \right),\\
    &\le 1-\exp \left( -\kappa^{\frac 1 3 } \vert \int_0^t \textbf{1}_{-[\frac \xi k ,\frac \eta l] \cup t- [\frac \xi k, \frac \eta l]}(\tau )d\tau \vert \right), \\
    &\le 1-\exp \left (-2\kappa^{\frac 1 3 } \vert \tfrac \xi k -\tfrac \eta l  \vert \right),\\
    &\lesssim  \kappa^{\frac 1 3 } \tfrac { \vert \xi l-k\eta\vert }{\vert kl \vert}.
\end{align*}
The case $M_\kappa(k,\xi)\le M_\kappa(l,\eta )$ follows from the same steps and $M_\kappa(k,\xi)\approx  M_\kappa(l,\eta )$.

\end{proof}

\section{Local Wellposedness }\label{sec:ex}
We expect the local wellposedness result to be well-known, but were not able to find it stated in the literature. In the following, we prove the local wellposedness by a standard application of the Banach fixed-point theorem.
\begin{pro}
    Consider equation \eqref{eq:p} with initial data $p_{in} \in H^N$ for $N\ge 5$. Then there exists a time $T$ such that there exists a unique solution $p(t)\in H^N$ to \eqref{eq:p} for all $t\in[0,T]$ .
\end{pro}

\begin{proof}
We prove existence with the Banach fixed-point theorem. Let  $T= 1+2\Vert p_{in}\Vert_{H^N}(1+\tfrac 8 \kappa)$ and let $X$ be the space 
\begin{align*}
    X=\{p\in L^\infty H^N\cap CH^{N-2} : \ p(t=0)=p_{in}, \  \Vert p\Vert_{L^\infty H^N}^2+\tfrac \kappa 2 \Vert \nabla_tp \Vert_{L^\infty H^N}^2 \le 2 \Vert p_{in} \Vert_{ H^N}^2\}  
\end{align*}
with the norm 
\begin{align*}
    \Vert p \Vert_X^2:= \Vert p\Vert_{L^\infty H^N}^2+\tfrac \kappa 2 \Vert \nabla_tp \Vert_{L^\infty H^N}^2. 
\end{align*}

We define $F:\ X\mapsto X$ as a mapping $q\mapsto p=F(q)$ such that $p$ solves
\begin{align*}
     \partial_t p_1 - \partial_x \partial_y^t \Delta^{-1}_t p_1- \alpha \partial_x p_2 &= \nu  \Delta_t p_1 +\Lambda^{-1}_t  \nabla^\perp_t (\nabla^\perp_t \Lambda_t^{-1} q_2 \nabla_t b- \nabla^\perp_t \Lambda_t^{-1} q_1\nabla_t v), \\
  \partial_t p_2 +\partial_x \partial_y^t \Delta^{-1}_t p_2 - \alpha \partial_x p_1 &= \kappa \Delta_t p_2  +\Lambda^{-1}_t \nabla^\perp_t (\nabla^\perp_t \Lambda_t^{-1} q_2\nabla_t v- \nabla^\perp_t \Lambda_t^{-1} q_1\nabla_t b),\\
  p|_{t=0}&= p_{in}. 
\end{align*}
Then the mapping $F$ satisfies:
\begin{enumerate}
    \item The mapping $F: \  X \to X $ is well defined on $X$.
    \item The mapping $F$ is a contraction, i.e.  $\Vert F(p)- F(\tilde p)\Vert_{L^\infty H^N}\le \tfrac 12  \Vert p-\tilde p\Vert_{L^\infty H^N}$.
\end{enumerate}
Since $X$ is a complete metric space, if we prove (1) and (2), then it follows that $F$ has a unique fixpoint by the Banach fixed-point theorem.

\begin{enumerate}
    \item Let $q\in X$, then we obtain for $p=F(q) $
    \begin{align*}
    \p_t \Vert p\Vert_{H^N}^2+ \kappa \Vert \nabla_t p\Vert_{H^N}^2 &\le \Vert p \Vert_{H^N }^2+\langle \Lambda^N v, \Lambda^N (\nabla^\perp_t \Lambda_t^{-1} q_2 \nabla_t b- \nabla^\perp_t \Lambda_t^{-1} q_1\nabla_t v)\rangle \\
    &\quad+ \langle \Lambda^N b, \Lambda^N (\nabla^\perp_t \Lambda_t^{-1} q_2 \nabla_t v- \nabla^\perp_t \Lambda_t^{-1} q_1\nabla_t b)\rangle \\
    &\le \Vert p \Vert_{H^N }^2 +\Vert p \Vert_{H^N}\Vert q \Vert_{H^N}\Vert \nabla_t  p^{n+1} \Vert_{H^N}\\
    &\le \Vert p \Vert_{H^N }^2 +\tfrac 2 \kappa \Vert p \Vert_{H^N}^2\Vert q \Vert_{H^N}^2 +\tfrac \kappa 2\Vert \nabla_t  p \Vert_{H^N}^2.
\end{align*}
Thus we obtain 
\begin{align*}
     \Vert p\Vert_{X}^2 
    &\le \Vert p_{in}  \Vert_{H^N }^2 + T(1+ \tfrac 1 \kappa \Vert q \Vert_{L^\infty H^N}^2) \Vert p\Vert_{L^\infty H^N}^2\\
    &\le \Vert p_{in}  \Vert_{H^N }^2 + T(1+ \tfrac 2 \kappa \Vert p_{in} \Vert_{H^N}^2) \Vert p\Vert_{L^\infty H^N}^2,
\end{align*}
Since 
\begin{align*}
    T(1+ \tfrac 4 \kappa \Vert p_{in} \Vert_{H^N}^2)<\tfrac 1 2 
\end{align*}
we infer the bound
\begin{align*}
     \Vert p\Vert_{X}^2 
    &\le 2 \Vert p_{in}  \Vert_{H^N }^2.
\end{align*}
As $\p_t p \in H^{N-2}$, it follows that $p\in CH^{N-2}$ and thus $p\in X$.

\item We show that $F$ is a contraction. Let $q,\tilde q \in X$ we denote  $p=F(q)$ and $\tilde p = F(\tilde q) $. We need to show that
\begin{align*}
    \Vert p-\tilde p \Vert_{X } <\tfrac 1 2  \Vert q-\tilde q \Vert_{X },
\end{align*}
by time estimate we obtain 
\begin{align*}
    \p_t \Vert p-\tilde p \Vert_{H^N}^2&+ \kappa \Vert \nabla_t (p- \tilde p)\Vert_{H^N}^2 \le \Vert p- \tilde p \Vert_{H^N }^2\\
    &\quad + \langle \Lambda^N (v-\tilde v), \Lambda^N (\nabla^\perp_t \Lambda_t^{-1} q_2 \nabla_t b- \nabla^\perp_t \Lambda_t^{-1} q_1\nabla_t v)\rangle\\
    &\quad + \langle \Lambda^N (b-\tilde b), \Lambda^N (\nabla^\perp_t \Lambda_t^{-1} q_2 \nabla_t v- \nabla^\perp_t \Lambda_t^{-1} q_1\nabla_t b)\rangle \\
    &\quad - \langle \Lambda^N (v-\tilde v), \Lambda^N (\nabla^\perp_t \Lambda_t^{-1} \tilde  q_2 \nabla_t  \tilde b- \nabla^\perp_t \Lambda_t^{-1}  \tilde q_1\nabla_t v)\rangle\\
    &\quad - \langle \Lambda^N (b-\tilde b), \Lambda^N (\nabla^\perp_t \Lambda_t^{-1} \tilde q_2 \nabla_t \tilde v- \nabla^\perp_t \Lambda_t^{-1} \tilde q_1\nabla_t \tilde b)\rangle \\
    &\le \Vert p- \tilde p \Vert_{H^N }^2+  \Vert p- \tilde p \Vert_{H^N }\left(  \Vert q- \tilde q \Vert_{H^N } \Vert \nabla_t \tilde p \Vert_{H^N }+ \Vert q \Vert_{H^N } \Vert \nabla _t (p- \tilde p) \Vert_{H^N } \right)\\
    &\le \Vert p- \tilde p \Vert_{H^N }^2(1+ \tfrac 2 \kappa\Vert q \Vert_{H^N } ^2)\\
    &\quad +  \Vert p- \tilde p \Vert_{H^N }  \Vert q- \tilde q \Vert_{H^N } \Vert \nabla_t \tilde p \Vert_{H^N } + \tfrac \kappa 2 \Vert \nabla _t (p- \tilde p) \Vert_{H^N }^2 .
\end{align*}
Integrating in time yields 
\begin{align*}
     \Vert p-\tilde p \Vert_{L^\infty H^N}^2&+ \tfrac \kappa 2 \Vert \nabla_t (p- \tilde p)\Vert_{L^2H^N}^2\\
    &\le  \Vert p- \tilde p \Vert_{L^\infty H^N }^2 T (1+ \tfrac 2 \kappa \Vert q \Vert_{L^\infty H^N } ^2)\\
    &\quad +  \sqrt T\Vert p- \tilde p \Vert_{L^\infty H^N }  \Vert q- \tilde q \Vert_{L^\infty H^N } \Vert \nabla_t \tilde p \Vert_{L^2 H^N }\\
    &\le  \Vert p- \tilde p \Vert_{L^\infty H^N }^2 T (1+\tfrac 2 \kappa \Vert q \Vert_{L^\infty H^N } ^2+ 4 \Vert \nabla_t \tilde p \Vert_{L^2 H^N }^2)\\
    &\quad + \tfrac T 4  \Vert q- \tilde q \Vert_{L^\infty H^N }^2 .
\end{align*}
Choosing $T$ such that 
\begin{align*}
    T (1+ \tfrac 2 \kappa \Vert q \Vert_{H^N } ^2+ \Vert \nabla_t \tilde p \Vert_{L^2 H^N }^2)&\le T + 2 T \Vert p_{in}\Vert_{H^N}(1+\tfrac 8 \kappa)<\tfrac 1 2,
\end{align*}
it follows, that
\begin{align*}
     \Vert p-\tilde p \Vert_{X}^2&\le \tfrac 1 2  \Vert p- \tilde p \Vert_{H^N } + \tfrac T 4  \Vert q- \tilde q \Vert_{L^\infty H^N }^2 .
\end{align*}
We hence conclude, that
\begin{align*}
     \Vert p-\tilde p \Vert_{X}^2&\le  \tfrac 1 2  \Vert q-\tilde q \Vert_{X}^2.
\end{align*}

\end{enumerate}

\end{proof}

\subsection*{Data availability}
No data was used for the research described in the article.

\subsection*{Acknowledgements}
Funded by the Deutsche Forschungsgemeinschaft (DFG, German Research Foundation) – Project-ID 258734477 – SFB 1173.
The author declares that they have no conflict of interest.
This article is part of the PhD thesis of Niklas Knobel.

\bibliography{library}

\begin{thebibliography}{RWXZ14}

\bibitem[BBZD23]{bedrossian21}
Jacob Bedrossian, Roberta Bianchini, Michele~Coti Zelati, and Michele Dolce.
\newblock Nonlinear inviscid damping and shear-buoyancy instability in the
  two-dimensional {B}oussinesq equations.
\newblock {\em Communications on Pure and Applied Mathematics},
  76(12):3685--3768, 2023.

\bibitem[BGM17]{bedrossian2017stability}
Jacob Bedrossian, Pierre Germain, and Nader Masmoudi.
\newblock On the stability threshold for the 3{D} {C}ouette flow in {S}obolev
  regularity.
\newblock {\em Annals of Mathematics}, pages 541--608, 2017.

\bibitem[BM14]{bedrossian2013asymptotic}
Jacob Bedrossian and Nader Masmoudi.
\newblock Asymptotic stability for the {C}ouette flow in the {2D} {E}uler
  equations.
\newblock {\em Applied Mathematics Research eXpress}, 2014(1):157--175, 2014.

\bibitem[BM15]{bedrossian2015inviscid}
Jacob Bedrossian and Nader Masmoudi.
\newblock Inviscid damping and the asymptotic stability of planar shear flows
  in the 2{D} {E}uler equations.
\newblock {\em Publ. Math. Inst. Hautes \'Etudes Sci.}, 122:195--300, 2015.

\bibitem[BSS88]{bardos1988longtime}
C~Bardos, Catherine Sulem, and P-L Sulem.
\newblock Longtime dynamics of a conductive fluid in the presence of a strong
  magnetic field.
\newblock {\em Transactions of the American Mathematical Society},
  305(1):175--191, 1988.

\bibitem[BVW18]{bedrossian2016sobolev}
Jacob Bedrossian, Vlad Vicol, and Fei Wang.
\newblock The {S}obolev stability threshold for {2D} shear flows near
  {C}ouette.
\newblock {\em Journal of Nonlinear Science}, 28(6):2051--2075, 2018.

\bibitem[CF23]{cobb2023elsasser}
Dimitri Cobb and Francesco Fanelli.
\newblock Els{\"a}sser formulation of the ideal {MHD} and improved lifespan in
  two space dimensions.
\newblock {\em Journal de Math{\'e}matiques Pures et Appliqu{\'e}es},
  169:189--236, 2023.

\bibitem[Dav16]{davidson_2016}
P.~A. Davidson.
\newblock {\em Introduction to Magnetohydrodynamics}.
\newblock Cambridge Texts in Applied Mathematics. Cambridge University Press, 2
  edition, 2016.

\bibitem[DM18]{dengmasmoudi2018}
Yu~Deng and Nader Masmoudi.
\newblock Long-time instability of the {C}ouette flow in low {G}evrey spaces.
\newblock {\em Communications on Pure and Applied Mathematics}, 2018.

\bibitem[Dol23]{Dolce}
Michele Dolce.
\newblock Stability threshold of the {2D} {C}ouette flow in a homogeneous
  magnetic field using symmetric variables.
\newblock {\em arxiv preprint arXiv:2308.12589}, 2023.

\bibitem[DZ21]{dengZ2019}
Yu~Deng and Christian Zillinger.
\newblock Echo chains as a linear mechanism: Norm inflation, modified exponents
  and asymptotics.
\newblock {\em Archive for rational mechanics and analysis}, 242(1):643--700,
  2021.

\bibitem[HHKL18]{hussain2018instability}
Zakir Hussain, Sultan Hussain, Tiantian Kong, and Zhou Liu.
\newblock Instability of {MHD} {C}ouette flow of an electrically conducting
  fluid.
\newblock {\em AIP Advances}, 8(10), 2018.

\bibitem[HT01]{hughes2001instability}
DW~Hughes and SM~Tobias.
\newblock On the instability of magnetohydrodynamic shear flows.
\newblock {\em Proceedings of the Royal Society of London. Series A:
  Mathematical, Physical and Engineering Sciences}, 457(2010):1365--1384, 2001.

\bibitem[HXY18]{he2018global}
Ling-Bing He, Li~Xu, and Pin Yu.
\newblock On global dynamics of three dimensional magnetohydrodynamics:
  nonlinear stability of {A}lfv{\'e}n waves.
\newblock {\em Annals of PDE}, 4(1):1--105, 2018.

\bibitem[IJ13]{ionescu2020nonlinear}
Alexandru~D Ionescu and Hao Jia.
\newblock Nonlinear inviscid damping near monotonic shear flows.
\newblock {\em Acta Mathematica}, 230(2):321 –-- 399, 2013.

\bibitem[Koz89]{kozono1989weak}
Hideo Kozono.
\newblock Weak and classical solutions of the two-dimensional
  magnetohydrodynamic equations.
\newblock {\em Tohoku Mathematical Journal, Second Series}, 41(3):471--488,
  1989.

\bibitem[KZ23a]{knobel2023echoes}
Niklas Knobel and Christian Zillinger.
\newblock On echoes in magnetohydrodynamics with magnetic dissipation.
\newblock {\em Journal of Differential Equations}, 367:625--688, 2023.

\bibitem[KZ23b]{knobel2023sobolev}
Niklas Knobel and Christian Zillinger.
\newblock On the {S}obolev stability threshold for the {2D} {MHD} equations
  with horizontal magnetic dissipation.
\newblock {\em arXiv preprint arXiv:2309.00496}, 2023.

\bibitem[Lis20]{liss2020sobolev}
Kyle Liss.
\newblock On the {S}obolev stability threshold of {3D} {C}ouette flow in a
  uniform magnetic field.
\newblock {\em Communications in Mathematical Physics}, pages 1--50, 2020.

\bibitem[LL11]{li2011resolution}
Y.~Charles Li and Zhiwu Lin.
\newblock A resolution of the {S}ommerfeld paradox.
\newblock {\em SIAM Journal on Mathematical Analysis}, 43(4):1923--1954, 2011.

\bibitem[MZ22]{masmoudi2022stability}
Nader Masmoudi and Weiren Zhao.
\newblock Stability threshold of two-dimensional {C}ouette flow in {S}obolev
  spaces.
\newblock {\em Ann. Inst. H. Poincar{\'e} C Anal. Non Lin{\'e}aire},
  39(2):245--325, 2022.

\bibitem[MZZ23]{masmoudi2023asymptotic}
Nader Masmoudi, Cuili Zhai, and Weiren Zhao.
\newblock Asymptotic stability for two-dimensional {B}oussinesq systems around
  the {C}ouette flow in a finite channel.
\newblock {\em Journal of Functional Analysis}, 284(1):109736, 2023.

\bibitem[RWXZ14]{ren2014global}
Xiaoxia Ren, Jiahong Wu, Zhaoyin Xiang, and Zhifei Zhang.
\newblock Global existence and decay of smooth solution for the 2-{D} {MHD}
  equations without magnetic diffusion.
\newblock {\em Journal of Functional Analysis}, 267(2):503--541, 2014.

\bibitem[Sch88]{schmidt1988magnetohydrodynamic}
Paul~G{\"u}nter Schmidt.
\newblock On a magnetohydrodynamic problem of {E}uler type.
\newblock {\em Journal of differential equations}, 74(2):318--335, 1988.

\bibitem[WZ17]{wei2017global}
Dongyi Wei and Zhifei Zhang.
\newblock Global well-posedness of the {MHD} equations in a homogeneous
  magnetic field.
\newblock {\em Analysis \& PDE}, 10(6):1361--1406, 2017.

\bibitem[Zil21]{zillinger2020boussinesq}
Christian Zillinger.
\newblock On the {B}oussinesq equations with non-monotone temperature profiles.
\newblock {\em Journal of Nonlinear Science}, 31(4):64, 2021.

\bibitem[ZZ23]{zhao2023asymptotic}
Weiren Zhao and Ruizhao Zi.
\newblock Asymptotic stability of {C}ouette flow in a strong uniform magnetic
  field for the {E}uler-{MHD} system.
\newblock {\em arXiv preprint arXiv:2305.04052}, 2023.

\end{thebibliography}
\bibliographystyle{alpha}

\end{document}